\documentclass[11pt]{amsart}
\usepackage{amssymb,amsmath}

\newcommand{\dist}{{\rm dist}}
\def\sgn{\operatorname{sgn}}

\def\Bbb{\mathbb}
\def\Cal{\mathcal}
\def\divv{\operatorname{div}}
\def\Dt{\partial_t}

\def\eb{\varepsilon}

\def\R {\mathbb{R}}

\def\<{\left<}
\def\>{\right>}

\def\Nx{\nabla_x}
\def\Dx{\Delta_x}
\def\({\left(}
\def\){\right)}

\newtheorem{proposition}{Proposition}[section]
\newtheorem{theorem}[proposition]{Theorem}
\newtheorem{corollary}[proposition]{Corollary}
\newtheorem{lemma}[proposition]{Lemma}

\theoremstyle{definition}
\newtheorem{definition}[proposition]{Definition}

\newtheorem{remark}[proposition]{Remark}

\numberwithin{equation}{section}

\def \no#1#2#3 {{\bf #1} (#3), #2.}
\def \eds#1#2#3 {#1, #2, #3.}

\title[Reaction diffusion equations] {Reaction-diffusion systems with supercritical nonlinearities  revisited}
\author[A. Kostianko, C. Sun and   S. Zelik]{ Anna Kostianko${}^{1,2}$, Chunyou Sun${}^2$
 and Sergey Zelik${}^{1,2}$}
\address{${}^1$
University of Surrey, Department of Mathematics,
Guildford, GU2 7XH, United Kingdom.}

\address{${}^2$ \phantom{e}School of Mathematics and Statistics, Lanzhou University, Lanzhou  \\ 730000,
P.R. China}
\email{anna.kostianko@surrey.ac.uk}
\email{s.zelik@surrey.ac.uk}
\email{sunchy@lzu.edu.cn}
\subjclass[2010]{35B40, 35B45,35K10}
\keywords{Reaction-diffusion system, monotone operators, convexity, long-time behavior, attractors,
 exponential attractors }
\begin{document}
\begin{abstract} We give a comprehensive study of  the analytic properties and long-time behavior
 of solutions of a reaction-diffusion system in a bounded domain in the case where the nonlinearity
  satisfies the standard monotonicity assumption. We pay the main attention to the supercritical case,
   where the nonlinearity is not subordinated to the linear part of the equation trying to put as small
    as possible amount of extra restrictions on this nonlinearity. The properties of such systems in the
     supercritical case may be  very different in comparison with the standard case of subordinated
      nonlinearities. We examine the global existence
     and uniqueness of weak and strong solutions, various types of smoothing properties,
     asymptotic compactness and the existence of global and exponential attractors.
\end{abstract}
\thanks{This work is partially supported by  the RSF grant   19-71-30004  as well as  the EPSRC
grant EP/P024920/1 and NSFC grants No. 11471148, 11522109, 11871169. }
\maketitle
\tableofcontents
\section{Introduction}\label{s0}
We study the following reaction-diffusion system in a bounded domain
 $\Omega\subset\R^d$ with smooth boundary:
\begin{equation}\label{0.rds}
\Dt u=a\Dx u-f(u)+g,\ \ u\big|_{t=0}=u_0
\end{equation}
endowed with the Dirichlet boundary conditions. Here $u=(u_1,\cdots,u_k)$ is an unknown
 vector-valued function, $a$ is a given diffusion matrix and $f(u)$ and $g$ are given nonlinearity
  and external forces respectively.
\par
Equations of the form \eqref{0.rds} model various classical phenomena in modern science
 (e.g., heat conduction, chemical kinetics, various quantum effects (Ginzburg-Landau equations),
  mathematical biology (Fitz-Hugh-Nagumo or Keller-Segel equations), etc.) and have been intensively studied from both
   mathematical and applied points of view, see \cite{BV,cv,CD19,DIL09,hale,lad,MP91,Rob,tem} and references therein. In a sense,
   this is the most studied and somehow simplest model example of an evolutionary PDE which may
   generate non-trivial dynamics.
\par
Since the analytic properties of the {\it linear} system \eqref{0.rds} are completely understood,
the analogous properties for the nonlinear equation depend strongly on whether or not we are able
 to treat the term $f(u)$ as a perturbation. As usual, if we want to have global existence of a solution,
  we need to find the proper {\it a priori estimates}, usually with the help of energy functionals or
   some "wisely" chosen Lyapunov type functionals. This, in turn, requires some restrictions on the
    function $f$ and matrix $a$ (to prevent the finite-time blow up of solutions). Then, if the found a
     priori estimates are strong enough to treat the nonlinearity as a perturbation
      (the so-called {\it subcritical} case), the analytic properties of the nonlinear equation is usually
       the same as for the dominating linear one and more or less complete theory is available. In contrast
        to this, in the {\it supercritical} case, the nonlinearity is strong enough to destroy the nice
        properties of the underlying linear equation, for instance, to produce the finite-time blow
         up of initially smooth solutions (despite the fact that the "energy" remains bounded
          and dissipative, see \cite{Bud} for such a phenomena in complex Ginzburg-Landau
           equation, \cite{P00} for chemical kinetics equations or \cite{HV97} for chemotaxis models). Usually,
            the sub/super criticality of the
            considered equation is determined by the {\it growth rate} of the nonlinearity $f(u)$ which
             depends on a priori estimates available (through the choice of the phase space for the problem)
             and the space dimension (through Sobolev embedding theorems). Thus, the typical picture
              for equation \eqref{0.rds} is the following: we have the so-called critical growth exponent
              $p=p_{crit}>1$  and an extra condition
              \begin{equation}\label{0.crit}
              |f(u)|\le C(1+|u|^p),\ u\in\R^k
              \end{equation}
on the nonlinearity and
the equation is subcritical if $p<p_{crit}$, critical if $p=p_{crit}$ and
 supercritical if $p>p_{crit}$, see \cite{BV,cv,Rob,tem} for more details.
 \par
Unfortunately, the universal conditions on $f$ and $a$ which would allow to avoid the finite-time blow up
and give the dissipativity in nice phase spaces are known in the scalar case $k=1$ only, so many
 different classes of sufficient conditions are suggested for the case of systems strongly depending
 on the area of science where the considered system comes from. For instance, from the point of view
  of chemical kinetics, it is natural to assume that $a$ is diagonal with non-negative entries and $f(u)$ satisfies
   the balance law
   \begin{equation}\label{0.bal}
   \sum_{i=1}^k f_i(u)\le 0
   \end{equation}
   which mimics the acting mass law for the concentrations $u_i$ of reagents (which usually belong to the
   non-negative cone in $\R^k$). The natural energy here is the $L^1$-norm of the solution $u(t)$ (the total
    mass is conserved or at least non-increasing), see \cite{MP91,P10} and the references therein for more details.
     We note  that in the supercritical case the solutions may blow up in finite time despite
      the conservation of total mass, see \cite{P00}.
  \par
 Clearly, assumptions \eqref{0.bal} are not appropriate for many other types of equations of the form
  \eqref{0.rds}, for instance, for complex Ginzburg Landau or Fitz-Hugh-Nagumo equations, so other types
   of assumptions should be used instead. The most widespread (especially in the
    literature related with the attractor theory, see \cite{BV,cv,Rob,tem}) is the following
     {\it dissipativity} condition:
    \begin{equation}\label{0.dis}
      f(u).u\ge -C,\ \ u\in\R^k
    \end{equation}
which is usually accompanied by the assumption that $a$ has a positive symmetric part. These assumptions
 are related with the so-called $L^2$-energy identity
\begin{equation}\label{0.l2}
\frac12\frac d{dt}\|u(t)\|^2_{L^2}+(a\Nx u,\Nx u(t))+(f(u(t)),u(t))=(g,u(t))
\end{equation}
which can be formally obtained by multiplying equation \eqref{0.rds} by $u$ and integrating
 over $x$ and which gives (due to these assumptions) the dissipative control of the
  $L^2$-norm of $u(t)$, see Lemma \ref{Lem1.l2}. However, the critical exponent which corresponds
   to this energy control (and the choice $H=L^2(\Omega)$ as a phase space):
   $$
   p_{crit}:=1+\frac4d
   $$
is rather restrictive (the most natural cubic nonlinearity is supercritical in 3D case) and not
 much can be said in general about the supercritical case where the uniqueness of solutions may be
  lost and finite-time blow up of the $L^\infty$-norm may occur (see \cite{Bud} for the numerical blow
   up evidence in 3D complex Ginzburg-Landau equation, see also \cite{cv,MZ} and references
    therein for study the long-time behavior of solutions without uniqueness using the
     multi-valued or trajectory approaches). We also mention here the so-called
      {\it anisotropic dissipativity} assumption:
     $$
     \sum_{i=1}^k f_i(u)u_i|u_i|^{l_i}\ge-C,
     $$
     where $l=(l_1,\cdots, l_k)$ is a sufficiently large vector, introduced in \cite{EZ03}. This restriction
     accompanied by the assumption that $a$ is diagonal gives $p_{crit}=\infty$ if $l=l(d)$ is large enough.
  \par
A natural alternative is to use the so-called {\it monotonicity} assumption:
\begin{equation}\label{0.mon}
f'(u)\ge-K
\end{equation}
which is also very widespread in the literature related with attractors. This assumption is
naturally related with the $H^1$-energy identity:
\begin{multline}\label{0.h1}
\frac12\frac d{dt}\|\Nx u(t)\|^2_{L^2}+(a\Dx u(t),\Dx u(t))+\\+(f'(u(t))\Nx u(t),\Nx u(t))=-(g,\Dx u(t))
\end{multline}
which is obtained by formal multiplication of \eqref{0.rds} by $-\Dx u$ and integration over $x$.
 Together with \eqref{0.mon}
this gives the dissipative control of the $H^1$-norm of the solution, see Lemma
 \ref{Lem1.h1} for the details. The critical growth exponent associated with this $H^1$-energy
  control is
 \begin{equation}\label{0.crit-h1}
p_{crit}=1+\frac4{d-2},\ d>2
 \end{equation}
can be found in many works, see \cite{BV,cv} and references therein. However, as pointed out
 in \cite{Z1}, the monotonicity assumption \eqref{0.mon} gives for free the control of $H^2$-norm of the
  solution $u(t)$ together with the $L^2$-norm of $f(u(t))$, namely, we have a priori estimates for the
   solutions in the nonlinear space
   $$
   \Bbb D:=\{u\in H^2(\Omega)\cap H^1_0(\Omega),\ f(u)\in L^2(\Omega)\}
   $$
   due to the control of the $L^2$-norm of $\Dt u(t)$, see Lemma \ref{Lem1.theta}
   and Corollary \ref{Cor1.d} below, and this gives us much better value of the critical exponent:
   \begin{equation}\label{0.d-crit}
p_{crit}=1+\frac4{d-4},\ d>4.
   \end{equation}
As far as we know, up to the moment, this is the best growth restriction which guarantees (of course,
under the monotonicity assumption \eqref{0.mon}) the global existence of smooth solutions and which is
 widely used nowadays not only for reaction-diffusion equations, but for many other related problems
  (like Cahn-Hilliard equations, see \cite{MZ} and references therein; strongly damped wave
   equations, see \cite{Chu12,KZ09} and reference therein, etc.).
   \par
We also note that the monotonicity assumption \eqref{0.mon} gives  the uniqueness of weak
solutions (= solutions in the energy phase space $H=L^2(\Omega)$, see section \ref{s5} below) even in the
 supercritical case which, in turn, allows to get a lot of information about the solutions and
 their long-time behavior in the supercritical case as well. The theory of equations \eqref{0.rds}
 in the critical or supercritical cases is of a great current interest,
  see for instance \cite{cv, CD03,CD19,MZ,Z03,ZS17} and references therein. However, in most cases
  rather essential
   extra restrictions on the nonlinearity $f$ are posed like the following two sided estimate:
\begin{equation}\label{0.strong}
   C(|u|^{p-1}-1)\ge f'(u)\ge -K+\alpha |u|^{p-1},\ \ u\in\R^k,\ \ C,\alpha>0
\end{equation}
which really simplifies the situation, but automatically excludes some interesting new phenomena which
 may appear in a general case.
 \par
The aim of the present paper (which can be considered as a
continuation of our work \cite{Z1}) is to give a comprehensive study of weak and strong solutions
 (=solutions in the phase space $\Bbb D$) of problem \eqref{0.rds} as well as their long-time behavior  in the supercritical case $p>p_{crit}$
with dissipative (assumption \eqref{0.dis} is fulfilled) and
 monotone (assumption \eqref{0.mon} is satisfied) nonlinearities
 trying to avoid/minimize  further restrictions on $f$.
\par
Weak and strong solutions of problem \eqref{0.rds} have been constructed in \cite{Z1}
 (see also section \ref{s5}). However,
in contrast to the case of assumptions \eqref{0.strong}, for   weak solutions the equation
is understood only in a
 sense of variational inequalities since we cannot guarantee that $f(u)\in L^1$ and, therefore, cannot treat the
  equation in the sense of distributions. By this reason, even the parabolic smoothing property (whether
   or not a weak solution becomes strong at the next time moment) becomes non-trivial and has been
   posed in \cite{Z1} as an open problem.
\par
Our first main result gives the positive answer on this question in the case where the nonlinearity has a
 polynomial growth rate.
 \begin{theorem}\label{Th0.main1} Let the nonlinearity $f$ satisfy the assumptions \eqref{0.crit}
  (for some $p>0$), \eqref{0.dis} and \eqref{0.mon}, the diffusion matrix have positive
   symmetric part and $g\in L^2(\Omega)$. Then any weak solution $u(t)$ starting
    from $u(0)\in H=L^2(\Omega)$ belongs to $\Bbb D$ for any $t>0$. In other words, the
     instantaneous parabolic $H$ to $\Bbb D$ smoothing property holds. In addition, the strong solutions of \eqref{0.rds}
      are dissipative in $\Bbb D$-norm as well.
 \end{theorem}
The proof of this theorem is based on estimation of $f(u)$ in Lebesgue spaces
 $L^q(\Omega)$ with $0<q<1$ and is given in section \ref{s6}.
\par
Our next result shows that the critical growth exponent can be slightly improved.

\begin{theorem}\label{Th0.main2} Let the assumptions of Theorem \ref{Th0.main1} hold and let,
in addition, the growth exponent $p$ of the nonlinearity satisfy
$$
p<p_{crit}+\eb=1+\frac4{d-4}+\eb,\ \ d>4
$$
for some small positive $\eb=\eb(a)$. Then any weak solution $u(t)$ of problem
 \eqref{0.rds} possesses an instantaneous
 $H$ to $L^\infty(\Omega)$ smoothing property. In particular, finite-time blow up of smooth solutions
 is impossible
 and the actual regularity of a solution $u(t)$ is restricted by the regularity of $\Omega$,
 $f$ and $g$ only. In the
  case where this data is $C^\infty$-smooth, the corresponding solution $u(t)$ will be also
  $C^\infty$ for any $t>0$.
\end{theorem}
We now turn to the attractors. The existence of a global attractor for problem \eqref{0.rds} in $H$ has
 been verified in \cite{Z1}, however, the question about strong attraction in $\Bbb D$ has been
  remained open. Our next result gives a positive answer on this question under the extra restriction
  \begin{equation}\label{0.fp}
|f'(u)|\le C(|f(u)|+1+|u|),\ \ u\in\R^k.
  \end{equation}
\begin{theorem}\label{Th0.main3} Let the assumptions of Theorem \ref{Th0.main1} hold and let,
in addition, $f$ satisfy \eqref{0.fp}. Then the solution semigroup $S(t)$ associated with
 problem \eqref{0.rds} possesses a compact global attractor $\Cal A$ in the phase space $\Bbb D$.
\end{theorem}
The proof of this theorem is given in section \ref{s7} and is based on the energy type arguments.
 We expect that assumption \eqref{0.fp} is technical and can be removed, but it is strongly related with
  the validity of the integration by parts formula
  $$
  (f(u),\Dx u)=-(f'(u)\Nx u,\Nx u),\ \ u\in\Bbb D,
  $$
see the discussion in section \ref{s9} below.
\par
Finally, we study the finite-dimensionality of the constructed global attractor $\Cal A$ in $\Bbb D$ and
 the existence
 of the so-called exponential attractor (see \cite{EFNT,EMZ00,EMZ05,MZ} and also section \ref{s8}
 for more details).

\begin{theorem}\label{Th0.main4} Let the assumptions of Theorem \ref{Th0.main3} hold and let, in addition, some extra
convexity assumptions on the function $u\to |f(u)|$ be posed (see formula \eqref{5.con-vex}). Then the solution
semigroup $S(t)$ associated with equation \eqref{0.rds} possesses an exponential attractor
 $\Cal M$ in $H$ and,
 in particular, the fractal dimension of the global attractor $\Cal A$ in $H$ is finite.
\end{theorem}
The finite-dimensionality of the global attractor $\Cal A$ has been established in \cite{Z1} under the
 similar assumptions using the so-called method of $l$-tra\-jec\-to\-ries developed in \cite{MN96,MP02}.
  In section
  \ref{s8} we suggest an alternative more transparent method for constructing of an
   exponential attractor which does not utilize $l$-trajectories and works directly in the phase space.
\par
The paper is organized as follows.
\par
Notations and spaces which will be used throughout of the paper are introduced in section \ref{s2} and
 the standard a priori estimates for the solutions $u(t)$ of problem \eqref{0.rds} are recalled
  in section \ref{s1.1}.
\par
The existence of strong solutions for problem \eqref{0.rds} is verified in section \ref{s.s} based on special
approximations of the nonlinearity $f$. The definition of a weak solution of problem \eqref{0.rds}
 in the sense of variational inequalities as well as the proof of its global existence and uniqueness is
 given in section \ref{s5}.
 Moreover, the existence of a global attractor $\Cal A$ for the solution semigroup $S(t)$ is also
  verified there.
\par
The weak to strong instantaneous smoothing property, see Theorem \ref{Th0.main1}, is verified in
 section \ref{s6}. The further regularity of strong solutions is obtained in section \ref{s7}.
  In particular, the proof of Theorems \ref{Th0.main2}
  and \ref{Th0.main3} are given there. Some results about the partial regularity of the
  elliptic problem associated with
  equations \eqref{0.rds} which have an independent interest are obtained in Appendix \ref{A}. The
  existence of an exponential attractor $\Cal M$, see Theorem \ref{Th0.main4}, is given in section \ref{s8}.
  \par
  Finally, section \ref{s9} discusses natural extensions of the developed theory to other classes
  of dissipative PDEs, in particular, to fractional reaction-diffusion systems and (fractional)
   Cahn-Hilliard type equations. At the end of this section we also discuss some important
    (at least from our point of view) open problems for further investigation.

\section{Assumptions and preliminaries}\label{s2}
Throughout of the paper we consider the following reaction-diffusion system in
 a bounded domain $\Omega\subset\R^d$:
\begin{equation}\label{1.main}
\Dt u=a\Dx u-f(u)+g,\ \ u\big|_{\partial\Omega}=0,\ \ u\big|_{t=0}=u_0.
\end{equation}
Here $u=(u_1,\cdots,u_k)$ is an unknown vector valued function, $a$ is a diffusion matrix satisfying
\begin{equation}\label{1.a}
a+a^*>0,
\end{equation}
$g\in L^2(\Omega)$ is a given external force and the nonlinearity $f$ is assumed
to satisfy the following conditions:
\begin{equation}\label{1.f}
\begin{cases}
1.\ f\in C^1(\R^k,\R^k),\\
2.\ f(u).u\ge -C, \ u\in\R^k, \\
3.\ f'(u)\ge -K,\ \ u\in\R^k,
\end{cases}
\end{equation}
where $C$ and $K$ are some fixed constants, $u.v$ stands for the standard inner product in $\R^k$
and $f'(u)\ge-K$ means $f'(u)\xi.\xi\ge-K|\xi|^2$ for all $\xi\in\R^k$.
\par
For any $l\in\Bbb N$ and any $1\le p\le \infty$ we denote by $W^{l,p}(\Omega)$ the
Sobolev space of distributions $u\in\Cal D'(\Omega)$ such that $u$ and all its partial
derivatives up to order $l$ inclusively belong to the Lebesgue space $L^p(\Omega)$. As usual,
for non-integer values of $l$, we define $W^{l,p}(\Omega):=B^l_{p,p}(\Omega)$ using real interpolation
 ($B^l_{p,p}$ is a classical Besov space, see e.g., \cite{tri}). Moreover, the symbol $W^{l,p}_0(\Omega)$
  stands for the closure of $C^\infty_0(\Omega)$ in $W^{l,p}(\Omega)$ and the space $W^{-l,p}(\Omega)$
   is defined as a dual space to $W^{l,p}_0(\Omega)$ with respect to the standard inner
    product in $H=L^2(\Omega)$. To simplify the notations, we will write $H^l(\Omega)$
     instead of $W^{l,2}(\Omega)$.
\par
In a sequel, we will also use the space $L^p(\Omega)$ with $0<p<1$ and use the standard notation
$$
\|u\|_{L^p(\Omega)}:=\(\int_\Omega |u(x)|^p\,dx\)^{1/p}
$$
simply ignoring the fact that it is not a norm. Recall that the topology in this space is defined by
 the metric $d_p(u,v):=\|u-v\|_{L^p}^p$.
 \par
We say that the function $u(t,x)$ is a {\it strong} solution of \eqref{1.main} if
\begin{equation}\label{1.u}
u\in C_w([0,T], H^2(\Omega)\cap H^1_0(\Omega)),\ \ f(u)\in C_w([0,T],L^2(\Omega))
\end{equation}
and equation \eqref{1.main} is satisfied in the sense of distributions. In particular, for strong solutions
 we require that the initial data $u_0\in \mathbb D$, where
\begin{multline}\label{1.d}
\mathbb D:=\{u\in H^2(\Omega)\cap H^1_0(\Omega),\ f(u)\in L^2(\Omega)\},\ \\
\|u\|^2_{\mathbb D}:=\|u\|^2_{H^2}+\|f(u)\|^2_{L^2}.
\end{multline}
Note that in general $\Bbb D$ is {\it not a linear} space and this  causes a lot of extra difficulties in
 comparison with the case of linear phase space. We define the topology in the space $\Bbb D$
  using the embedding
  $$
  j:\Bbb D\to [H^2(\Omega)\cap H^1_0(\Omega)]\times L^2(\Omega),\ \ j(u)=\{u,f(u)\}.
  $$
In particular, the sequence $u_n\to u$ strongly in $\Bbb D$ if $u_n\to u$ in $H^2(\Omega)$
 and $f(u_n)\to f(u)$ in $L^2(\Omega)$. Analogously, we say that $u_n\rightharpoondown u$ {\it weakly}
 in $\Bbb D$ if $u_n\to u$ weakly in $H^2(\Omega)$ and $f(u_n)\to f(u)$ weakly in $L^2(\Omega)$.

\section{A priori estimates}\label{s1.1}
In this section, we give a number of more or less standard estimates for strong solutions of
 problem \eqref{1.main}
which will be justified later. We start with the dissipative estimate in the space $H=L^2(\Omega)$.

\begin{lemma}\label{Lem1.l2} Let $g\in H$,  assumptions \eqref{1.a}-\eqref{1.f} hold and let $u$ be
a sufficiently smooth solution of \eqref{1.main}. Then, the following estimate is valid:
\begin{multline}\label{1.l2}
\|u(T)\|^2_{L^2}+\int_T^{T+1}\|u(t)\|^2_{H^1}\,dt+\int_T^{T+1}|f(u).u|\,dt\le\\\le
 Ce^{-\alpha T}\|u_0\|^2_{L^2}+C(\|g\|^2_{L^2}+1),
\end{multline}
where the positive constants $C$ and $\alpha$ are independent of $t$ and $u_0$.
\end{lemma}
\begin{proof} We multiply equation \eqref{1.main} by $u$ and integrate over $x$. This gives
$$
\frac12\frac d{dt}\|u\|^2_{L^2}+(a\Nx u,\Nx u)+(f(u),u)=(g,u),
$$
where $(u,v):=\int_\Omega u(x).v(x)\,dx$ is a standard inner product in $H$. Using the
 dissipativity assumption $f(u).u\ge-C$ and positivity of the matrix $a$ together with the
 Friedrichs inequality,
 we arrive at
$$
\frac12\frac d {dt}\|u\|^2_{L^2}+\alpha\|u\|^2_{L^2}+\alpha\|\Nx u\|^2_{L^2}+
(|f(u).u|,1)\le C(\|g\|^2_{L^2}+1)
$$
for some positive constants $C$ and $\alpha$. The
 Gronwall inequality applied to this relation gives \eqref{1.l2} and finishes the proof of the lemma.
\end{proof}
The next lemma gives the analogous dissipative estimate for the $H^1$-norm of the solution.
\begin{lemma}\label{Lem1.h1} Let the assumptions of Lemma \ref{Lem1.l2} hold and $u$ be a
 sufficiently regular solution of \eqref{1.main}. Then, the following estimate is valid:
\begin{multline}\label{1.h1}
\|u(T)\|_{H^1}^2+\int_T^{T+1}\|u(t)\|_{H^2}^2\,dt+\\+
\int_T^{T+1}(|f'(u)\Nx u(t).\Nx u(t)|,1)\,dt\le Ce^{-\alpha T}\|u_0\|^2_{H^1}+C(\|g\|^2_{L^2}+1),
\end{multline}
where the positive constants $C$ and $\alpha$ are independent of $t$ and $u_0$.
\end{lemma}
\begin{proof} We multiply equation \eqref{1.main} by $-\Dx u$ and integrate over $x$ to get
$$
\frac12\frac d{dt}\|\Nx u\|^2_{L^2}+(a\Dx u,\Dx u)+(f'(u)\Nx u,\Nx u)=-(g,\Dx u)
$$
Using the inequality $f'(u)\ge-K$ and positivity of matrix $a$ again, we arrive at
\begin{multline}
\frac d{dt}\|\Nx u\|^2_{L^2}+\alpha\|\Nx u\|^2_{L^2}+\alpha\|\Dx u\|^2_{L^2}+\\+(|f'(u)\Nx u.\Nx u|,1)
\le C(\|g\|^2_{L^2}+\|\Nx u\|^2_{L^2}).
\end{multline}
Applying the Gronwall inequality to this relation and using \eqref{1.l2} in order to control the integral
of $\|\Nx u\|^2_{L^2}$, we end up with \eqref{1.h1} and finish the proof of the lemma.
\end{proof}
Let now $\theta=\Dt u$. Then this function solves
\begin{equation}\label{1.theta}
\Dt\theta=a\Dx\theta-f'(u)\theta,\ \ \theta\big|_{t=0}=a\Dx u_0-f(u_0)+g.
\end{equation}
The next lemma gives the $L^2$-estimate for the time derivative $\theta$.
\begin{lemma}\label{Lem1.theta} Let the assumptions of Lemma \ref{Lem1.l2} hold and let
$u$ be a sufficiently regular solution of equation \eqref{1.main}. Then the following estimate is valid:
\begin{multline}\label{1.l2h}
\|\theta(T)\|^2_{L^2}+\int_T^{T+1}\|\theta(t)\|_{H^1}^2\,dt+\\+
\int_{T}^{T+1}(|f'(u)\theta(t).\theta(t)|,1)\,dt\le C\|u_0\|_{\mathbb D}^2e^{K_1T}+Ce^{K_1 T}(\|g\|^2+1),
\end{multline}
where positive constants $C$ and $K_1$ are independent of $t$ and $u_0$.
\end{lemma}
\begin{proof} We multiply equation \eqref{1.theta} by $\theta$ and use assumption
$f'(u)\ge-K$ and positivity of matrix $a$
to get
$$
\frac d{dt}\|\theta\|^2_{L^2}+\alpha\|\Nx \theta\|^2_{L^2}+(|f'(u)\theta.\theta|,1)\le
 2K\|\theta\|^2_{L^2}.
$$
Applying the Gronwall inequality to this relation, we get the desired estimate and finish
the proof of the lemma.
\end{proof}
As a corollary of this lemma, we get the key control for the norm of the solution in the space $\Bbb D$.
\begin{corollary}\label{Cor1.d} Let the assumptions of Lemma \ref{Lem1.l2} hold and let
 $u$ be a sufficiently regular solution of \eqref{1.main}. Then the following estimate is valid:
\begin{equation}\label{1.D}
\|u(T)\|_{\mathbb D}^2\le Ce^{K_1 T}\|u_0\|^2_{\mathbb D}+Ce^{K_1T}(1+\|g\|^2_{L^2}),
\end{equation}
where the positive constants $C$ and $K_1$ are independent of $T$ and $u_0$.
\end{corollary}
\begin{proof} We rewrite equation \eqref{1.main} as an elliptic problem
\begin{equation}\label{1.ell}
a\Dx u(T)-f(u(T))=\tilde g(T):=\theta(T)-g
\end{equation}
for every fixed $T$. Multiplying then this equation by $\Dx u(T)$ (without integration in time!)
 and using the control for $\theta(T)$
 obtained above (together with the elliptic
 regularity estimate for the Laplacian and the assumption $f'(u)\ge-K$), we arrive at the estimate
$$
\|u(T)\|_{H^2}^2\le C(\|g\|^2_{L^2}+\|\theta(T)\|_{L^2}^2+\|\Nx u(T)\|_{H}^2).
$$
Expressing the $L^2$-norm of $f(u)$ from equation \eqref{1.ell}, we get
$$
\|u(T)\|^2_{\Bbb D}\le C(\|g\|^2_{L^2}+\|\theta(T)\|_{L^2}^2+\|\Nx u(T)\|_{H}^2).
$$
Estimating the right-hand side of this inequality by \eqref{1.l2h} and \eqref{1.h1} and using that
$$
\theta(0)=\Dt u(0)=a\Dx u(0)-f(u(0))+g,
$$
we arrive at the desired estimate and finish the proof of the corollary.
\end{proof}
\begin{remark}
Note that, in contrast to estimates for the $L^2$ and $H^1$
norms of the solution $u(t)$, the obtained estimate for the $\mathbb D$-norm of $u(t)$ is {\it not}
dissipative and even growth exponentially in time. We will remove this drawback
 later (under some extra assumptions on $f$).
\end{remark}
We conclude this section by establishing the global Lipschitz continuity
 with respect to the initial data which plays a crucial role in constructing
  weak solutions for \eqref{1.main}.
\begin{lemma}\label{Lem1.lip} Let the assumptions of Lemma \ref{Lem1.l2} hold and let $u_1(t)$
and $u_2(t)$ be two sufficiently regular solutions of equation \eqref{1.main}. Then the
 following estimate is valid:
\begin{multline}\label{1.lip}
\|u_1(T)-u_2(T)\|^2_{L^2}+\\+\int_0^T\|u_1(t)-u_2(t)\|^2_{H^1}\,dt\le Ce^{K_1 T}\|u_1(0)-u_2(0)\|^2_{L^2},
\end{multline}
where the positive constants $C$ and $K_1$ are independent of $T$, $u_1$ and $u_2$.
\end{lemma}
\begin{proof} Indeed, let $v(t)=u_1(t)-u_2(t)$. Then this function solves
\begin{equation}
\Dt v=a\Dx v-[f(u_1)-f(u_2)].
\end{equation}
Multiplying this equation by $v$, using that, due to the monotonicity assumption $f'(u)\ge -K$,
$$
[f(u_1)-f(u_2)].[u_1-u_2]\ge -K|v|^2,
$$
and arguing as in the proof of Lemma \ref{Lem1.theta},
we arrive at the desired estimate.
\end{proof}

\section{Existence of strong solutions}\label{s.s}

 Although the construction of a solution from a priori estimates obtained above is more
 or less standard, it is a bit delicate here since the phase
 space $\mathbb D$ is in general {\it nonlinear} and, particularly, it is not clear whether or not
  smooth functions are dense in $\Bbb D$. By this reason, we sketch the proof here.
\par
We expect that the existence result can be also obtained using the monotone operators theory
(e.g., the standard Ioshida approximations), but we prefer to give an alternative,
 a bit more transparent proof.
  Our idea is to approximate the nonlinearity $f$ by a sequence $f_n$ of functions of sub-linear growth
  rate without destroying assumptions \eqref{1.f}. Then, on the one hand, the existence of
 solutions for such $f_n$ is well-known and, on the other hand, as not difficult to see,
 all estimates obtained above will be uniform with respect to $n$. Thus, it will only
  remain to pass to the limit $n\to\infty$. We start with the approximation of $f$.

\begin{lemma}\label{Lem1.fn} Let the function $f$ satisfy assumptions \eqref{1.f}.
 Then, there exists a sequence of functions $f_n\in C^1(\R^k,\R^k)$ such that
\begin{equation}\label{1.fn}
1.\ \ f_n(u).u\ge-C,\ \ 2.\ \ f_n'(u)\ge-K
\end{equation}
uniformly with respect to $n$. Moreover,
\begin{equation}
f_n\to f
\end{equation}
in $C_{loc}(\R^k,\R^k)$ and
 \begin{equation}
 |f_n'(u)|\leq C_n,\ \ u\in\R^k,
\end{equation}
where the constant $C_n$ may depend on $n$.
\end{lemma}
\begin{proof}[Sketch of the proof] Let us first introduce a smooth  scalar convex
function $\Psi(z)$ (we may also require that $\Psi'(s)\geq 0$ as $s\geq 0$)
in such a way that $\nabla_u\Psi(|u|^2)$ grows faster than $|f(u)|$ as $|u|\to\infty$. Then,
 for every fixed $\eb>0$, we
consider the function $\bar f_\eb(u)=f(u)+\eb\nabla_u\Psi(|u|^2)$. Since the second term will dominate
 the first one if $|u|^2$ is large enough, introducing the first cut-off function $\theta_R(|u|^2)$ such
 that $\theta_R=1$ is $|u|\le R$ and zero if $|u|\ge 2R$, we may find
find $R=R_\eb\ge\frac1\eb$ such that the function
$$
\bar f_{\varepsilon} (u):=\theta_{R_\eb}(|u|^2)f(u)+\ \eb \nabla_u \Psi(|u|^2)
 $$
 satisfies assumptions \eqref{1.f} uniformly with respect to $u\in \mathbb{R}^k$.
  Note that $\bar f_\eb(u)= \eb \nabla_u \Psi(|u|^2)= \eb \Psi'(|u|^2)u$ for $|u|^2\ge2R_\eb$
   and we may make it linear for $|u|^2\ge3R_\eb$ by cutting-off
   $\Psi''$ on the interval $2R_\eb\leq |u|^2\le3R_\eb$. For instance, we may
    introduce another  cut-off function  $\tilde{\theta}_{R_\eb}(\tau)=1$ as
    $\tau\leq 2R_\eb$
    and zero if $\tau\geq 3R_\eb$ and define
    $$
    \Psi'_{R_\eb}(s):=\int_0^{s}\tilde{\theta}_{R_\eb }(\tau)\Psi''(\tau)d\tau+\Psi'(0).
    $$
    This gives
    $$
    f_\eb(u)=\theta_{R_\eb}(|u|^2)f(u)+ \eb \nabla_u\Psi_{R_\eb}(|u|^2).
    $$
Taking finally $\varepsilon =\varepsilon _n=\frac{1}{n}$, we get the desired approximating sequence.
\end{proof}
We now introduce the approximating system for \eqref{1.main}
\begin{equation}\label{1.main-app}
\Dt u=a\Dx u-f_n(u)+g,\ \ u\big|_{t=0}=u_0^n,\ \ u\big|_{\partial\Omega}=0,
\end{equation}
where the functions $f_n$ are constructed in Lemma \ref{Lem1.fn}. However, the choice of
 the approximating initial data $u_0^n$ requires some accuracy. Indeed, we cannot
 just fix $u_0^n=u_0$ since $\|f_n(u_0)\|_{L^2}$ will be not uniformly bounded and, as
 a result, we may lose the estimate of the $\mathbb D$-norm of the limit solution. Instead, we
  define $u_0^n$ as a solution of the following auxiliary elliptic problem:
\begin{equation}\label{1.ell1}
a\Dx v-f_n(v)-Kv=G:=a\Dx u_0+f(u_0)-Ku_0,\ \ v\big|_{\partial\Omega}=0.
\end{equation}
The next lemma gives useful properties of the solutions of this auxiliary problem.
\begin{lemma}\label{Lem1.au} Let the functions $f_n$ be as above and $u_0\in\mathbb D$. Then,
for every fixed $n$, problem \eqref{1.ell1} has a unique solution $v=u_0^n$. Moreover,
$\|u_0^n\|_{H^2}$ and $\|f_n(u_0^n)\|_{L^2}$ are uniformly bounded as $n\to\infty$ and
\begin{equation}
u_0^n\rightharpoondown u_0,\ \ f_n(u_0^n)\rightharpoondown f(u_0)
\end{equation}
in the spaces $H^2$ and $L^2$ respectively.
\end{lemma}
\begin{proof} Indeed, the existence and uniqueness of
 a solution for \eqref{1.ell1} is obvious since $f_n(u)$
are monotone and have sublinear growth
rate. Let us prove uniform bounds.
 Indeed, multiplying \eqref{1.ell1}
 by $\Dx v=\Dx u_0^n$ and using the monotonicity, we get the estimate
$$
\|u_0^n\|_{H^2}^2\le C\|G\|^2_{L^2}\le C\|u_0\|^2_{\mathbb D},
$$
so $u_0^n$ is uniformly bounded in $H^2$. Expressing $f_n(v)$
from equation \eqref{1.ell1}, we see that $f_n(u_0^n)$ are also uniformly bounded.
\par
Let us verify the convergence. Since $u_0^n$ is uniformly bounded, passing to a subsequence if
necessary, we may assume that
$u_0^n\rightharpoondown w$ as $n\to\infty$ and $u_0^n\to w$ strongly to $w$ in $H^1$.
 Then, we have the convergence $u_0^n(x)\to w(x)$ almost everywhere. Moreover,
from this convergence and Lemma \ref{Lem1.fn}, we may conclude   that $f_n(u_0^n(x))\to f(w(x))$
almost everywhere. Since $f_n(u_0^n)$ are uniformly bounded in $L^2$, passing
 to a subsequence again, we infer that $f_n(u_0^n)\rightharpoondown f(w)$.
Passing after that to the weak limit $n\to\infty$
 in equations \eqref{1.ell1}, we see that the limit function $w$ solves
\begin{equation}\label{1.lim}
a\Dx w-f(w)-Kw=G=a\Dx u_0-f(u_0)-Ku_0,\ \ w\big|_{\partial\Omega}=0.
\end{equation}
Finally, since the solution $w\in\mathbb D$ of equation \eqref{1.lim} is unique (again
 due to the monotonicity of $f(u)+Ku$), we conclude that $w=u_0$. The uniqueness
 also gives that passing to a subsequence was not necessary and the whole
  sequence $u_0^n$ converges to $u_0$.
\end{proof}
We are now ready to state and prove the main result of this section.
\begin{theorem}\label{Th1.well} Let the nonlinearity $f$ and matrix $a$ satisfy
assumptions \eqref{1.f} and \eqref{1.a},
$g\in L^2$ and $u_0\in\mathbb D$. Then, problem \eqref{1.main} possesses a unique strong
solution $u(t)\in\mathbb D$ which satisfies all estimates formally obtained in section \ref{s1.1}.
\end{theorem}

\begin{proof} We approximate the desired solution $u$ by the approximate solutions
 $u_n(t)$ of problems \eqref{1.main-app}, where $f_n$ and $u_0^n$ are chosen as
  in Lemmas \ref{Lem1.fn} and \ref{Lem1.au}. Then, since $f_n$ has a sublinear growth,
   the existence and uniqueness of a solution $u_n(t)$ of \eqref{1.main-app} is straightforward.
    At the next step, we need to check that all estimates of section \ref{s1.1} are indeed
    uniform with respect to $n$ (the justification of all these estimates for the case of
    sublinear growth rate is also obvious). This is obvious for
Lemmas \ref{Lem1.l2} and \ref{Lem1.h1} (as well as for Lemma \ref{Lem1.lip}) since
$u_0^n\to u_0$ {\it strongly} in $H^1$ and $f_n$ satisfy \eqref{1.f} uniformly with
respect to $n$. Thus, we only need to look on the estimates related with time
differentiation and $\mathbb D$ norms. The key role in these estimates is played
by the $L^2$-norm of time derivative $\theta_n(t):=\Dt u_n(t)$ and the $L^2$-norm of
it at time moment $t$ is estimated by its $L^2$-norm at time $t=0$.
But due to our construction
\begin{equation}\label{1.nice}
\theta_n(0)=a\Dx u_0^n-f_n(u_0^n)+g=a\Dx u_0-f(u_0)+g+K(u_0^n-u_0).
\end{equation}
Therefore, according to Lemma \ref{Lem1.au},
\begin{equation}\label{1.nice1}
\|\theta_n(0)\|_{L^2}^2\le C(\|u_0\|^2_{\mathbb D}+\|g\|^2_{L^2}+1).
\end{equation}
By this reason, the analogue of estimate \eqref{1.D} on the level of approximations reads
\begin{equation}\label{1.Dn}
\|\Dt u_n(t)\|^2_{L^2}+\|u_n(t)\|^2_{H^2}+\|f_n(u_n(t))\|^2_{L^2}\le
 Ce^{K_1t}(\|u_0\|^2_{\mathbb D}+\|g\|^2_{L^2}+1),
\end{equation}
where positive constants $C$ and $K_1$ are independent of $t$, $u_0$ and $n$.
\par
When the uniform estimates are obtained, we may pass to the limit $n\to\infty$ in equations
 \eqref{1.main-app} and construct the desired strong solution $u(t)$ of the limit problem
  \eqref{1.main} (the passage to the limit in the nonlinear term is done exactly as
   in Lemma \ref{Lem1.au}. The uniqueness of a solution is an immediate corollary of
   Lemma \ref{Lem1.lip} (which does not require justification on the level of strong solutions).
   Finally, passing to the limit in the corresponding estimates for $u_n$,
    we prove that the limit solution $u(t)$ satisfies indeed all of the estimates of
    section \ref{s1.1}. The only non-immediate thing is the passage to the limit in
    the terms like $(f_n'(u_n)\Nx u_n,\Nx u_n)$ (and in the analogous term containing $\Dt u_n$) since
     we do not have any control of the integral norms of $f_n'(u_n)$. However, the passage to
     the limit could be performed here using the condition $f'_n(u_n)\ge-K$,
 the fact that $f_n'(u_n(t,x))\to f'(u(t,x))$ and the convexity arguments.
 Namely, under these conditions, we may establish that
\begin{multline}\label{1.connvex}
 \int_T^{T+1}(f'(u(t))\Nx u(t),\Nx u(t))\,dt\le\\\le
 \liminf_{n\to\infty}\int_T^{T+1}(f_n'(u_n(t))\Nx u_n(t),\Nx u_n(t))\,dt
\end{multline}
 (see e.g \cite{Ball}, Theorem 5.4). Thus, the theorem is proved.
\end{proof}

\section{Weak solutions, dissipativity and attractors}\label{s5}

In the previous section, we have proved the global existence and uniqueness
of {\it strong} solutions of \eqref{1.main}. Thus, the solution semigroup
\begin{equation}\label{1.sem}
S(t):\mathbb D\to\mathbb D,\ \ S(t)u_0:=u(t)
\end{equation}
is well-defined. Moreover, according to Lemma \ref{Lem1.lip}, this semigroup
 is globally Lipschitz continuous in the $L^2$-metric:
\begin{equation}\label{1.s-lip}
\|S(t)u_0^1-S(t)u_0^2\|_{L^2}^2\le Ce^{K_1t}\|u_0^1-u_0^2\|_{L^2}^2,\ \ u_0^1,u_0^2\in\mathbb D.
\end{equation}
Thus, we can extend this semigroup by continuity from $\mathbb D$ to its closure in $L^2$ which obviously
 coincides with the whole $L^2$ (since $C^\infty\subset \mathbb D$). Thus, the semigroup
\begin{equation}\label{1.sem-w}
\widehat S(t)u_0:=\lim_{n\to\infty}S(t)u_0^n,\ \ u_0^n\in \mathbb D,\ \ u_0=\lim_{n\to\infty}u_0^n
\end{equation}
is well-defined. Moreover, the limit in \eqref{1.sem-w} can be considered in the space
$C(0,T;L^2(\Omega))$, so the trajectories $u(t):=\widehat S(t)u_0$
 automatically belong to the space $C(0,T;L^2(\Omega))$ for all $T\ge0$.
\par
Our next step is to understand in what sense the trajectory $u(t)$ thus
constructed satisfies the initial equation \eqref{1.main}. Note that, for the general
 $u_0\in L^2$, we do not know whether or not $f(u)\in L^1$, so we cannot treat it
 in the distributional sense. Indeed, the only control related with $f(u)$ which we have up to now
  follows from estimate \eqref{1.l2} and claims that $f(u).u\in L^1(\Omega)$ (being pedantic, this
   is proved for the strong solutions only, but it can be easily extended to
    weak solutions using the Fatou lemma). Unfortunately, this is not
     enough to control the $L^1(\Omega)$-norm of the function $f$ itself, so we cannot treat
     the term $f(u)$ in a distributional sense.
\par

 Instead, we use the ideas from the monotone operator theory and
  variational inequalities, see e.g., \cite{Brezis}. Namely, following \cite{EGZ},
   we take an arbitrary test function
\begin{equation}\label{1.test}
v\in C_w(0,T;\mathbb D)\cap C^1_w(0,T;L^2(\Omega)),\ \ \forall T\in\R_+,
\end{equation}
and multiply formally equation \eqref{1.main} by $u(t)-v(t)$ and integrate
over $[0,T]$. Then, integrating by parts and using that
$$
-(a\Dx u-a\Dx v,u-v)\ge0,\ \ (f(u)-f(v),u-v)\ge -K\|u-v\|^2_{L^2},
$$
we end up with
\begin{multline}\label{1.var}
\frac12\|u(T)-v(T)\|^2_{L^2}-\frac12\|u(0)-v(0)\|^2_{L^2}+\\+\int_0^T(\Dt v(t),u(t)-v(t))\,dt\le\\\le
 \int_0^T(a\Dx v(t)-f(v(t))+g,u(t)-v(t))\,dt+K\int_0^T\|u(t)-v(t)\|^2_{L^2}\,dt.
\end{multline}
The advantage of this approach is that the variational inequality \eqref{1.var} has a sense for all
 $u\in C(0,T;L^2(\Omega))$ and, therefore, can be used to define
 a weak solution $u(t)$ of problem \eqref{1.main}.

\begin{definition}\label{Def1.weak} A function $u\in C(0,T;L^2(\Omega))$, $T\in\R_+$ is a weak solution
 of problem \eqref{1.main} if $u(0)=u_0$ and the variational inequality \eqref{1.var}
 holds for every $T\in\R_+$ and every test function $v$ satisfying \eqref{1.test}.
\end{definition}

We are now ready to state the key result of this section.

\begin{theorem}\label{Th1.weak} Let $g\in L^2(\Omega)$, the diffusion matrix $a$ be such
that $a+a^*>0$ and $f$ satisfy assumptions \eqref{1.f}. Then,
for every $u_0\in L^2(\Omega)$, there exists a unique weak solution $u(t)$ of problem \eqref{1.main}
 and this solution has the form $u(t)=\widehat S(t)u_0$, where the extension
  $\widehat S(t)$ of the solution semigroup $S(t)$ is defined by \eqref{1.sem-w}.
\end{theorem}

\begin{proof} Indeed, any {\it strong} solution is a weak solution of \eqref{1.main}
 (since all manipulations used in the derivation of the variational inequality
 \eqref{1.var} are obviously justified on the level of strong solutions).
 Let now $u(t)=\lim_{n\to\infty}u_n(t)$, where $u_n(t)$ are the strong solutions $u_n(t):=S(t)u_0^n$,
 $u_0^n\in\mathbb D$ and $u_0^n\to u_0$ in $L^2$. The variational inequality for $u_n$ reads
\begin{multline}\label{1.varn}
\frac12\|u_n(T)-v(T)\|^2_{L^2}-\frac12\|u_n(0)-v(0)\|^2_{L^2}+\\+\int_0^T(\Dt v(t),u_n(t)-v(t))\,dt\le\\\le
 \int_0^T(a\Dx v(t)-f(v(t))+g,u_n(t)-v(t))\,dt+K\int_0^T\|u_n(t)-v(t)\|^2_{L^2}\,dt.
\end{multline}
Using that $u_n\to u$ in $C(0,T;L^2(\Omega))$ and
 passing to the limit $n\to\infty$ in \eqref{1.varn}, we see that $u(t)$
  satisfies \eqref{1.var} and, therefore,
 $u(t):=\widehat S(t)u_0$ is the desired weak solution of \eqref{1.main}.
\par
Vice versa, let $\bar u(t)$ be a weak solution of \eqref{1.main} and let $\bar u(t):=\widehat S(t)\bar u(0)$.
 Then, there exists a sequence $u_0^n\in\mathbb D$ such that
$u_0^n\to \bar u(0)$ and the sequence of {\it strong} solutions $u_n(t)=S(t)u_0^n$ which
 converges as $n\to\infty$ to the weak solution $u(t)$. We need to show that $u(t)=\bar u(t)$.
\par
Indeed, by the definition of a strong solution, $u_n(t)$ satisfies the assumptions of \eqref{1.test} and
 therefore can be used as a test function $v=u_n$ in the variational inequality \eqref{1.var}
 for the weak solution $\bar u$. Taking $v=u_n$ in it and using that $\Dt u_n=a\Dx u_n-f(u_n)+g$, we get
$$
\frac12\|\bar u(T)-u_n(T)\|_{L^2}^2\le
K\int_0^T\|\bar u(t)-u_n(t)\|^2_{L^2}\,dt+\frac12\|\bar u(0)-u_n(0)\|^2_{L^2}.
$$
Passing to the limit $n\to\infty$ in this inequality, we get
$$
\|\bar u(T)-u(T)\|^2_{L^2}\le 2K\int_0^T\|\bar u(t)-u(t)\|^2_{L^2}\,dt
$$
and, since $T$ is arbitrary, the Gronwall
 inequality gives that $\bar u(t)=u(t)$ for all $t$. Thus, the theorem is proved.
\end{proof}

\begin{remark}\label{Rem.w-sol} There is an alternative possibility to relate a
 weak solution $u(t)$ with equation \eqref{1.main}, namely, for every $\psi\in C_0^\infty(\R^k,\R^k)$,
  the identity
\begin{equation}\label{def-cut}
\Dt(\psi(u))=\psi'(u)a\Dx u-\psi'(u)f(u)+\psi'(u)g
\end{equation}
should be satisfied in a sense of distributions, see \cite{Z1}. It is not difficult to verify that,
 indeed, any weak solution $u(t)$ should satisfy this identity. The drawback of this approach is
  that it is unclear whether or not \eqref{def-cut} is enough for the uniqueness.
  \par
  Using identity \eqref{def-cut} it is not difficult to show that under the additional restriction
\begin{equation}\label{1.fdot}
  |f(u)|\le C(|f(u).u|+|u|+1)
\end{equation}
which guarantees that $f(u)\in L^1([0,T]\times\Omega)$, any weak solution satisfies
  equation \eqref{1.main} in a sense of distributions. However, in contrast to the scalar case $k=1$,
  in the
   case of systems \eqref{1.fdot} is an extra restriction which we prefer to avoid.
\end{remark}

As a next step, we note that the weak solutions are dissipative. Indeed, passing
 to the limit in the estimate of Lemma \ref{Lem1.l2} for strong solutions, we derive that
\begin{equation}\label{1.l2dis}
\|\widehat S(T)u_0\|_{L^2}^2+\int_T^{T+1}\|\widehat S(t)u_0\|^2_{H^1}\,dt\le
 Ce^{-\alpha T}\|u_0\|^2_{L^2}+C(1+\|g\|^2_{L^2})
\end{equation}
which is a standard dissipative estimate for the semigroup $\widehat S(t)$.
 Analogously, passing to the limit
 in the estimate of Lemma \ref{Lem1.h1}, we get the dissipative estimate in $H^1$ for weak solutions
\begin{equation}\label{1.h1dis}
\|\widehat S(T)u_0\|_{H^1}^2+\int_T^{T+1}\|\widehat S(t)u_0\|^2_{H^2}\,dt\le
 Ce^{-\alpha T}\|u_0\|^2_{H^1}+C(1+\|g\|^2_{L^2}).
\end{equation}
In addition, estimates \eqref{1.l2dis} and \eqref{1.h1dis} give in a standard way the
 $L^2$-$H^1$ smoothing property
 for the semigroup $\widehat S(t)$, namely, the following estimate holds:
\begin{equation}\label{1.smo}
\|\widehat S(t)u_0\|_{H^1}^2\le C\frac{t+1}t\(e^{-\alpha t}\|u_0\|_{L^2}^2+1+\|g\|^2_{L^2}\).
\end{equation}
These estimates ensure us that the ball
 $\mathcal B_R:=\{u\in H^1_0,\ \|u\|_{H^1}\le R\}$ will be a compact
(in $L^2$) {\it absorbing} ball for the semigroup $\widehat S(t)$
 if $R=R(\|g\|_{L^2})$ is large enough. Remind that the latter means that for
every bounded set $B$ of $L^2$ there exists a time $T=T(B)$ such that
$$
\widehat S(t)B\subset \mathcal B_R
$$
for all $t\ge T$. This fact, in turn, allows us to establish the existence of a
 global attractor for the solution semigroup $\widehat S(t)$ in the phase space $H=L^2(\Omega)$.
 We recall that, by definition, a set $\mathcal A\subset H$ is a global attractor
 for a semigroup $\widehat S(t): H\to H$ if the following conditions are satisfied:
\par
1) $\mathcal A$ is a {\it compact} subset of $H$;
\par
2) $\mathcal A$ is strictly invariant: $\widehat S(t)\mathcal A=\mathcal A$, for all $t\ge0$;
\par
3) It attracts the images of bounded sets as $t\to\infty$, namely, for any bounded $B\subset H$
 and any neighbourhood $\mathcal O(\mathcal A)$, there exists $T=T(B,\mathcal O)$ such that
$$
\widehat S(t)B\subset\mathcal O(\mathcal A)
$$
for all $t\ge T$.
\par
The next theorem may be considered as the second key result of this section.
\begin{theorem}\label{Th1.attr} Let the assumptions of Theorem \ref{Th1.weak} be satisfied. Then
 the weak solution semigroup $\widehat S(t)$ possesses a global attractor $\mathcal A$
 in $H=L^2(\Omega)$ which is a bounded set of $H^1(\Omega)$ and possesses the following description:
\begin{equation}\label{1.str}
\mathcal A=\mathcal K\big|_{t=0},
\end{equation}
where $\mathcal K\subset C_b(\R,L^2(\Omega))$ is the set of all complete bounded
trajectories of the semigroup $\widehat S(t)$:
\begin{equation}\label{1.K}
\mathcal K:=\{u\in C_b(\R,L^2(\Omega)),\ u(t+h)=\widehat S(t)u(h),\ h\in\R,\ t\in\R_+\}.
\end{equation}
\end{theorem}
\begin{proof} According to the abstract attractor existence theorem,
 see e.g., \cite{BV}, we need to verify two assumptions:
\par
1) The operators $\widehat S(t): H\to H$ are continuous for every fixed $t$;
\par
2) The semigroup $\widehat S(t)$ possesses a compact absorbing set in $H$.
\par
The first assumption is guaranteed by Lemma \ref{Lem1.lip} and the second one
 follows from estimate \eqref{1.smo}. Thus,
 the global attractor $\mathcal A$ exists. The fact that $\mathcal A$ is a bounded subset of $H^1$ follows
 from the fact that the attractor is always  a subset of an absorbing set and the representation formula
 \eqref{1.str} is a standard corollary of the attractor existence theorem. Thus, the theorem is proved.
\end{proof}

\section{Weak to strong smoothing property}\label{s6}

In this section, we establish that any weak solution $u(t)$ of problem \eqref{1.main} becomes strong
 for $t>0$. The main difficulty here is the fact that we cannot in general estimate
  $|f(u)|$ through $f(u).u$ and, by this reason, we do not know whether
  or not $f(u)$ and $\Dt u$ are distributions. This makes the situation with the parabolic
   smoothing property
  a bit more delicate than usual. We overcome this difficulty under the extra assumption that
\begin{equation}\label{1.fp}
|f(u)|\le C(1+|u|^p)
\end{equation}
for some $p>1$ by using the $L^q$-spaces with $q<1$. Namely, we will use the fact that
\begin{equation}
\|f(u)\|_{L^{2/p}}\le C(\|u\|^{p}_{L^2}+1),
\end{equation}
where, for $q<1$, we denote by $\|v\|_{L^q}$ exactly the same expression as for the case
$q\ge1$ (simply ignoring the fact that it is no more a norm). Thus, at least on the level of approximations,
we may expect that, for $\theta=\Dt u$,
\begin{equation}\label{1.good}
\|\theta(t)\|_{L^2(0,1;L^2(\Omega))+L^\infty(0,1;L^{2/p}(\Omega))}\le
C(\|u_0\|_{H^1}^{p}+\|g\|^{p}_{L^2}+1)
\end{equation}
and this can be used in order to establish the smoothing property for $\theta$. Namely,
 the following theorem holds.
\begin{theorem}\label{Th1.strange}Let $d\ge3$,  the assumptions of Theorem \ref{Th1.weak} and \eqref{1.fp}
hold, and let $u_0\in H^1$ and $u(t)$ be the corresponding weak solution of equation \eqref{1.main}.
Then, $u(t)\in\mathbb D$ for all $t>0$ and the following estimate is valid:
\begin{equation}\label{1.smm}
\|\Dt u(t)\|_{L^2}^2+\|u(t)\|^2_{\mathbb D}\le Ct^{-N}\(Q(\|u_0\|_{H^1})+Q(\|g\|_{L^2})\),\ \ t\in(0,1],
\end{equation}
where the exponent $N=N(p)$, positive constant $C$ and the monotone increasing function $Q$
are independent of $u_0$ and $t$.
\end{theorem}
\begin{proof} We first note that it is enough to verify \eqref{1.smm} for $u_0\in\mathbb D$
only when $u(t)$ is a {\it strong} solution. Moreover, it is enough to obtain the estimate
 for $\Dt u$ only since then the estimate for $\|u(t)\|_{\mathbb D}$ will follow from the
  elliptic problem \eqref{1.ell}. Second, we approximate the strong solution $u(t)$ by the
  solutions $u_n(t)$ of auxiliary problems \eqref{1.main-app}. Finally, analyzing the proof
  of Lemma \ref{Lem1.fn}, we see that assumption \eqref{1.fp} allows us to pose the
   extra assumption
\begin{equation}\label{1.fext}
|f_n(u)|\le C(1+|u|^{p_1})
\end{equation}
for some $p_1>p>1$ and constant $C$ independent of $n$ (e.g., by taking $\Psi(u):=|u|^{p_1+1}$).
\par
Let $\theta_n(t):=\Dt u_n(t)$. Then, this function solves the equation
\begin{equation}
\Dt\theta_n=a\Dx\theta_n-f_n'(u_n)\theta_n.
\end{equation}
Multiplying this equation by $t^N\theta_n(t)$ where $N$ is a sufficiently big number, we end up with
\begin{multline}\label{1.sm-t}
\frac d{dt}(t^N\|\theta_n(t)\|^2_{L^2})+\alpha t^N\|\theta_n(t)\|^2_{H^1}-
\\-2K(t^N\|\theta_n(t)\|^2_{L^2})\le C t^{N-1}\|\theta_n(t)\|^2_{L^2}.
\end{multline}
We need to estimate the integral in the right-hand side. To this end, we fix
sufficiently small $s>0$ which will be specified below and write
\begin{multline}\label{1.split}
\|\theta_n\|^2_{L^2}=(|\theta_n|^2,1)=(|\theta_n|^s,|\theta_n|^{2-s})\le\\\le
C(|\Dx u_n|^s,|\theta_n|^{2-s})+C(|g|^s,|\theta_n|^{2-s})+C(|f_n(u_n(s))|^s,|\theta_n|^{2-s}),
\end{multline}
where we have used equation \eqref{1.main-app} in order to express $\theta_n$
through $u_n$. Let us estimate every term in the RHS separately. Applying the H\"older
 and Young inequalities to the first term, we get
\begin{multline}\label{1.split1}
t^{N-1}(|\Dx u_n|^s,|\theta_n|^{2-s})\le C\|\Dx u_n\|_{L^2}^s\(t^{\frac{N-1}{2-s}}
\|\theta_n\|_{L^2}\)^{2-s}\le\\\le C\|\Dx u_n\|^2_{L^2}+t^{2\frac{N-1}{2-s}}\|\theta_n\|^2_{L^2}.
\end{multline}
This gives us a good estimate if $N$ is chosen in such a way that $2\frac{N-1}{2-s}\ge N$.
The second term in the RHS of \eqref{1.split} can be estimated analogously to have
\begin{equation}\label{1.split2}
t^{N-1}(|g|^s,|\theta_n|^{2-s})\le  C\|g\|^2_{L^2}+t^{2\frac{N-1}{2-s}}\|\theta_n\|^2_{L^2}.
\end{equation}
Let us now estimate the most complicated third term. To this end, we use the embedding theorem
 $H^1\subset L^r$ where $\frac1r=\frac12-\frac1d$ together with the H\"older inequality with
 exponents $q_1$ and $q_2$, $\frac1{q_1}+\frac1{q_2}=1$, with $(2-s)q_2=r$ to get
$$
(|f_n(u_n)|^s,|\theta_n|^{2-s})\le\|f_n(u_n)\|_{L^{sq_1}}^s\|\theta_n\|_{H^1}^{2-s}.
$$
Moreover, due to assumptions  \eqref{1.fext}, we have
$$
\|f_n(u_n)\|_{L^{2/p_1}}\le C(\|u_n\|_{L^2}+1)^{p_1}\le C(\|u_0\|_{L^2}+1+\|g\|_{L^2})^{p_1},
$$
where $C$ is independent of $n$. We may also fix $sq_1=\frac2{p_1}$ to get
\begin{equation}\label{1.big}
(|f_n(u_n)|^s,|\theta_n|^{2-s})\le C(\|u_0\|_{L^2}+1+\|g\|_{L^2})^{sp_1}\|\theta_n\|^{2-s}_{H^1}
\end{equation}
and end up with the following system for the exponents $q_1$, $q_2$ and $s$
$$
\frac1{q_1}+\frac1{q_2}=1,\ \ \frac1{q_1}=\frac{sp_1}{2},\ \ \frac1{q_2}=(2-s)\(\frac12-\frac1d\).
$$
Solving this system, we get
$$
s=\frac{4}{d(p_1-1)+2},\ \ q_1=\frac{d(p_1-1)+2}{2p_1}
$$
and we see that $0<s<2$ and $1<q_1<\infty$, so all of the exponents are in the prescribed range
and \eqref{1.big} holds indeed. Applying the Young inequality, we arrive at
\begin{equation}\label{1.split3}
t^{N-1}(|f_n(u_n)|^s,|\theta_n|^{2-s})\le Q_\eb(\|u_0\|_{L^2})+Q_\eb(\|g\|_{L^2})+
\eb t^{2\frac{N-1}{2-s}}\|\theta_n\|^2_{H^1}
\end{equation}
where $\eb>0$ is arbitrary small and the monotone function $Q_\eb$ is independent of $u_0$.
Combining estimates \eqref{1.split}, \eqref{1.split1}, \eqref{1.split2} and \eqref{1.split3}
for estimating the RHS of \eqref{1.sm-t} and fixing $\eb>0$ to be small enough and
$N$ satisfying $2\frac{N-1}{2-s}\ge N$, we arrive at
\begin{multline}\label{1.fin}
\frac d{dt}\(t^N\|\theta_n(t)\|^2_{L^2}\)\le K_1\(t^N\|\theta_n(t)\|^2_{L^2}\)+\\+
C\|\Dx u_n(t)\|^2_{L^2}+Q(\|u_0^n\|_{H^1})+Q(\|g\|_{L^2}),\ \ t\le1.
\end{multline}
Applying the Gronwall inequality to this relation and using \eqref{1.h1dis} for estimating the
 integral of the $H^2$-norm of the solution, we end up with
\begin{equation}\label{1.fin1}
t^N\|\theta_n(t)\|^2_{L^2}\le Q(\|u_0^n\|_{H^1})+Q(\|g\|_{L^2}),\ \ t\in(0,1]
\end{equation}
and passing to the limit $n\to\infty$, we derive the desired estimate for $\Dt u(t)$. Thus,
 the theorem is proved.
\end{proof}

Combining this result with the $L^2$ to $H^1$ smoothing property \eqref{1.smo}, we get the following result.

\begin{corollary}\label{Cor1.smm} Under the assumptions of Theorem \ref{Th1.strange}, the weak
solution semigroup $\widehat S(t)$ possesses the following smoothing property:
\begin{equation}\label{1.sm-fin}
\|\widehat S(t)u_0\|_{\mathbb D}^2\le C t^{-N-1}\(Q(\|u_0\|_{L^2})+Q(\|g\|_{L^2})\), \ t\in(0,1],
\end{equation}
where the positive constant $C$ and monotone function $Q$ are independent of $t$ and $u_0\in L^2(\Omega)$.
\end{corollary}
Thus, under the  assumption \eqref{1.fp}, any weak solution  indeed becomes strong
 for $t>0$ (the extra assumption $d\ge3$ is not essential since for $d\le2$ the equation is
  subcritical and the smoothing property is obvious). This, in particular, gives the following result
   on the regularity of the global attractor.

\begin{corollary}\label{Cor1.att-sm} Let the assumptions of Theorem \ref{Th1.strange} hold. Then the global
 attractor $\mathcal A$ of the solution
semigroup $\widehat S(t)$ constructed in Theorem \ref{Th1.attr} is a bounded subset of $\mathbb D$:
\begin{equation}\label{1.att-D}
\|\mathcal A\|_{\mathbb D}\le Q(\|g\|_{L^2})
\end{equation}
for some monotone increasing function $Q$.
\end{corollary}

Indeed, this assertion is an immediate corollary of \eqref{1.sm-fin} and the strict invariance of
 the global attractor.

\section{Further regularity and strong attraction}\label{s7}

In this section we discuss the possibility to get more regular than $u\in\Bbb D$ solutions.
 We start with some partial result on the regularity of $\Dt u$ which does not require any extra
  assumptions on $f$ and $g$.

\begin{proposition}\label{Prop5.der} Let the assumptions of Theorem \ref{Th1.well} hold. Then,
 there exists
a positive number $r>0$ depending only on the matrix $a$ such that, for any strong solution
 $u(t)\in\Bbb D$,
the following estimate holds:
\begin{equation}\label{5.smdt}
t\|\Dt u(t)\|_{L^{r+2}}^{r+2}+\int_0^ts\|\Nx(|\Dt u(s)|^{\frac{r+2}2})\|^2_{L^2}\,ds\le C
\|\Dt u(0)\|_{L^2}^{r+2},
\end{equation}
where $0\le t\le1$ and the constant $C$ is independent of $t$ and $u$.
\end{proposition}
\begin{proof} Let $\theta:=\Dt u$. Then this function satisfies equation \eqref{1.theta}. Let us multiply
 this equation by $\theta|\theta|^r$ and integrate over $x$. This gives
 $$
 \frac1{r+2}\frac d{dt}\|\theta(t)\|^{r+2}_{L^{r+2}}-(a\Dx\theta,\theta|\theta|^r)+
 (f'(u)\theta,\theta|\theta|^r)=0.
$$
Integrating by parts in the second term, we get
$$
-(a\Dx\theta, \theta|\theta|^r)\ge (a\Nx\theta,\Nx\theta|\theta|^r)-
Cr(|\Nx\theta|^2,|\theta|^r)\ge (\alpha-Cr)(|\Nx\theta|^2,|\theta|^r)
$$
for some positive $\alpha$. Fixing now $r>0$ small enough and estimating the term
 containing $f$ using $f'(u)\ge-K$, we arrive at
 $$
\frac1{r+2}\frac d{dt}\|\theta(t)\|^{r+2}_{L^{r+2}}+\frac\alpha2(|\Nx\theta|^2,|\theta(t)|^r)\le
K\|\theta(t)\|^{r+2}_{L^{r+2}}.
 $$
 Multiplying this estimate by $t$ and integrating in time,
 we arrive at
\begin{equation}
t\|\theta(t)\|^{r+2}_{L^{r+2}}+\int_0^ts(|\Nx\theta(s)|^2,|\theta(s)|^r)\,ds\le
C\int_0^t\|\theta(s)\|^{r+2}_{L^{r+2}}\,ds
\end{equation}
for $0\le t\le1$. To estimate the right-hand side of this inequality, we use estimate \eqref{1.l2h} and
Sobolev embedding theorem
which gives that, for sufficiently small $r>0$,
$$
\int_0^t \|\theta(s)\|^{r+2}_{L^{r+2}}\,ds\le C\|\theta\|_{L^\infty(0,t;L^2)\cap L^2(0,t;H^1)}^{r+2}\le
 C\|\theta(0)\|_{L^{2}}^{r+2}
$$
and finishes the proof of the proposition.
\end{proof}

\begin{corollary} Let the assumptions of Proposition \ref{Prop5.der} hold and let in addition equation
\eqref{1.main} be dissipative in $\Bbb D$, i.e.
\begin{equation}
\|u(t)\|_{\Bbb D}\le Q(\|u(0)\|_{\Bbb D})e^{-\alpha t}+Q(\|g\|_{L^2})
\end{equation}
for some monotone function $Q$. Then every trajectory $u(t)$, $t\in\R$ belonging to the kernel
 $\Cal K$ possesses the following extra regularity of time derivative:
\begin{equation}\label{5.dra}
 \|\Dt u(t)\|_{L^{r+2}}\le Q(\|g\|_{L^2}),\ \ t\in\R
\end{equation}
 for some $r>0$ and monotone function $Q$ which is independent of $t$ and $u$.
\end{corollary}
This extra regularity in time can be transformed to extra regularity in space assuming that
 right-hand side $g$ is slightly more regular.

 \begin{corollary}\label{Prop5.derx} Let the assumptions of Proposition \ref{Prop5.der} hold and let,
  in addition,
  \begin{equation}\label{5.g}
  g\in L^{q}(\Omega)
  \end{equation}
for some $q>2$. Then, there exists $r=r(a,q)>0$ such that
\begin{equation}\label{5.xdir}
t\|\Nx u(t)\|_{L^{\frac{d(r+2)}{d-2}}}^{r+2}\le C\(\|u(0)\|_{\Bbb D}+\|g\|_{L^q}\)^{r+2}
\end{equation}
for $0\le t\le1$. In particular, if the attractor $\Cal A$ is a bounded set in $\Bbb D$ then it
 is also bounded in $W^{1,\frac{d(r+2)}{d-2}}(\Omega)$.
 \end{corollary}
Indeed, due to Proposition \ref{Prop5.der}, we control the $L^{r+2}$-norm of $\Dt u$. Rewriting
problem \eqref{1.main} as an elliptic
boundary value problem
$$
a\Dx u(t)-f(u(t))=\tilde g(t):=\Dt u(t)-g,
$$
we get also the control for the $L^{r+2}$-norm of $\tilde g(t)$ (point-wisely in time). Applying
the elliptic regularity result proved in Appendix (see Theorem \ref{ThA.main}) to this equation, we arrive
 at the desired estimate \eqref{5.xdir}.
\par
The obtained partial regularity results allow us to establish the crucial $L^\infty$-estimates
for critical and
 slightly supercritical growth rate of the nonlinearity $f$. Namely, the following result holds.

 \begin{theorem}\label{Th5.crit} Let the assumptions of Proposition \ref{Prop5.der} hold and let,
 in addition,  the nonlinearity
  $f$ satisfy \eqref{1.fp} with the exponent $p$ restricted by the assumption
  \begin{equation}\label{5.crit}
p<p_{crit}+\eb,\ \ p_{crit}:=1+\frac4{d-4},\ \eb=\eb(a)>0
  \end{equation}
  if $d\ge4$
  and the external forces $g$ satisfy \eqref{5.g} for some $q>\frac d2$. Then any weak solution $u(t)$ of
   problem \eqref{1.main} possesses the following smoothing property:
   \begin{equation}\label{5.inf}
   \|u(t)\|_{L^\infty}\le Q_t(\|u(0)\|_{L^2}+\|g\|_{L^q}),\ \ t\in(0,1]
   \end{equation}
   for some monotone function $Q_t$ depending on $t$, but  independent of $u$ and $g$.
\end{theorem}

\begin{proof}Note that, due to estimate \eqref{1.sm-fin} we may assume from the very beginning that
 $u_0\in\Bbb D$ and work with strong solutions only. The derivation of \eqref{5.inf} can be done
  by the standard bootstrapping arguments by iterating the classical interior regularity result for the
  linear parabolic equation
  \begin{equation}
\Dt u-a\Dx u=h(t),
  \end{equation}
  namely,
  \begin{equation}\label{5.linreg}
\|u\|_{L^\infty(T,T+1;W^{2-\kappa,s})}\le C_{T,s,\kappa}\(\|u\|_{L^2(0,T+1;L^2)}+
\|h\|_{L^\infty(0,T+1;L^s)}\)
  \end{equation}
  where $\kappa>0$ is arbitrarily small, $1<s<\infty$ and $T>0$. This estimate, in turn, can be easily
   deduced from the fact that this linear equation generates an analytic semigroup in $L^s(\Omega)$ or from
   the maximal $L^q(t,t+1;L^s)$-regularity estimate for parabolic equations (see e.g., \cite{tri}),
    so we left the details to the reader).
\par
From this smoothing property and Sobolev embedding theorem, we derive the iterative estimate
\begin{equation}
\|u\|_{L^\infty(T_{k+1},1;L^{q_{k+1}})}\le C_k\(\|u\|_{L^2(0,1;L^2)}+\|h\|_{L^\infty(T_k,1;L^{s_k})}\)
\end{equation}
where $T_{k+1}>T_k$ and
$$
q_{k+1}:=\min\left\{\infty, \frac{s_k d}{d-s_k(2-\kappa)}\right\}.
$$
In our situation $h(t)=g-f(u(t))$ and, due to our growth restrictions, we have
$$
\|h\|_{L^{s_k}}\le C(\|g\|_{L^q}+1+\|u\|^p_{L^{q_k}})
$$
where $s_k:=\min\{q,q_kp^{-1}\}$. Thus, in order to prove the theorem, it is sufficient
 to verify that the sequence $q_k$ defined via
$$
q_0=\frac{d(r+2)}{d-r-4},\ \ q_{k+1}=\frac{q_k d}{pd-q_k(2-\kappa)},\ \ \kappa\ll1
$$
will become large than $q>\frac d2$ in finitely many steps (we have used here estimate
 \eqref{5.xdir} and the embedding $W^{1,\frac{d(r+2)}{d-2}}\subset L^{q_0}$
 to initialize the iterations and the embedding $W^{2-\kappa,q}\subset L^\infty$ which holds for
  sufficiently small $\kappa$
 due to the condition $q>\frac d2$).
 \par
 Obviously this sequence will be monotone increasing if (and only if)
 $$
p-\frac{q_0}d(2-\kappa)<1.
 $$
 Then it must converge to $+\infty$, so we only need to verify the last inequality. Using assumption
  \eqref{5.crit} and the
  explicit formula for $q_0$, we only need the inequality
  $$
  \frac4{d-4}+\eb-(2-\kappa)\frac{r+2}{d-r-4}<0.
  $$
It remains to note that the last inequality is satisfied if $\kappa\ll1$ and $\eb<\eb_0=\eb_0(r)$ for
some positive $\eb_0$ if $r>0$. This finishes the proof of the theorem.
\end{proof}

\begin{remark}\label{Rem5.scrit} The growth rate of the nonlinearity is no more important
 if the $L^\infty$-estimate for
the solutions is obtained, so further regularity can be obtained by bootstrapping exactly
 as in the subcritical case. Thus, under the growth restriction \eqref{5.crit}, the actual regularity of
  a solution is determined by the smoothness of $\Omega$, $f$ and $g$ only (if all of them are
   $C^\infty$-smooth,
   the solutions will be also $C^\infty$-smooth). In other words, we may say that the critical growth
   exponent for $f$  in our problem \eqref{1.main} is slightly {\it larger}
    than $p_{crit}=1+\frac4{d-4}$.
    We also note that the value $\eb=\eb(a)$ somehow measures how far the matrix $a$ is  from
    the scalar matrix.
    It is easy to show that $\eb(a)=\infty$ if $a$ is scalar.
\end{remark}

We now turn to the question of whether or not the attraction to $\Cal A$ holds in the space $\Bbb D$.
Since in this
 case we at least need the dissipativity of our semigroup in $\Bbb D$, we assume that $f$ has
  a polynomial growth rate (i.e., that \eqref{1.fp} is satisfied for some $p\in\R_+$). Of course, the most
  interesting here is the supercritical case when the assumption \eqref{5.crit} is not satisfied.
   Unfortunately, we do not know the answer on this question in general and have to pose some
   extra restrictions which however look natural. Namely, we assume that
    the nonlinearity also satisfies
    \begin{equation}\label{5.fprime}
|f'(u)|\le C(1+|f(u)|+|u|),\ \ u\in\R^k.
    \end{equation}
Then, the following result holds.
\begin{theorem}\label{Th5.D} Let the assumptions of Proposition \ref{Prop5.der} hold and let also
 assumptions \eqref{1.f} and \eqref{5.fprime} be satisfied. Then the image $\widehat S(1)B_R$ of
  any closed ball $B_R$ of
  radius $R$ in $H$ is a compact set in $\Bbb D$. In particular the global attractor $\Cal A$ is
  compact in $\Bbb D$
   and attracts in the strong topology of $\Bbb D$ as well.
\end{theorem}
\begin{proof} We only need to prove the compactness of $\hat S(1)B_R$, the rest is a corollary of
 the standard
 attractor's existence theorem. The fact that this set is closed is also standard and
  we left it to the reader. So, we will only check pre-compactness below.
  \par
  The proof of this fact is a combination of parabolic regularity estimates which gives the
   pre-compactness of the set
   $$
   \Cal B_R:=\{\Dt u(1),\, u(t)=\widehat S(t)u_0,\ u_0\in B_R\}
   $$
    in $L^2(\Omega)$
   and energy type estimates for the elliptic equation which then give the desired compactness
   in $H^2(\Omega)$.
\par
{\it Step 1. $\Cal B_R$ is compact in $L^2(\Omega)$.} We already know that $\Cal B_R$ is a
bounded set in $L^{2+r}(\Omega)$,
 due to Proposition \ref{Prop5.der} and Corollary \ref{Cor1.smm}. In order to get the desired compactness
  we will use the standard interpolation embedding:
\begin{equation}\label{5.int-in}
 W^{1-\kappa,1}(\Omega)\cap L^{r+2}(\Omega)\subset H^{(1-\kappa)\frac r{2(r+1)}}(\Omega),
\end{equation}
  see \cite{tri}. This embedding together with the compactness of the embedding $H^{\eb}\subset L^2$ will
   give the desired result if we prove boundedness of $\Cal B_R$ in $W^{1-\kappa,1}$
    for some $0<\kappa<1$. To this end, we note that according to Corollary \ref{Cor1.smm},
    $u(t)\in\Bbb D$ for $t\in[1/2,1]$ and is uniformly bounded there if $u_0\in B_R$. Thus, according
     to assumption \eqref{5.fprime}, $\|f'(u(t))\|_{L^2}$ is uniformly bounded. Since the $L^{2+r}$-norm
      of $\Dt u(t)$ is also bounded due to Proposition \ref{Prop5.der}, we have the estimate
      $$
      \|f'(u)\Dt u\|_{L^s}\le C_R,\ \ t\in[1/2,1],\ \ \frac 1s=\frac12+\frac1{r+2}.
      $$
Thus, applying the  $L^s$ interior estimate (with  $s>1$) to equation
$$
\Dt\theta-a\Dx\theta=-f'(u(t))\Dt u(t),
$$
we arrive at
\begin{multline}
\|\theta(1)\|_{W^{2(1-\frac1s),s}}\le C(\|\Dt\theta\|_{L^s(3/4,1;L^s)}+\|\Dx\theta\|_{L^s(3/4,1;L^s)})
\le\\\le
 C(\|f'(u)\Dt u\|_{L^s(1/2,1;L^s)}+\|\Dt u\|_{L^2(1/2;1;L^2)})\le C_R'.
\end{multline}
This estimate gives the desired boundedness of $\Cal B_R$ in $W^{1-\kappa,1}(\Omega)$ and completes
 the first step of the proof.
 \par
 {\it Step 2. Compactness in $H^2$.} Let us consider a sequence of solutions $u_n(t)$, $u_n(0)\in B_R$ and
 find a subsequence which is convergent strongly in $H^2$ to some solution $u(t)$. Due to the
  result of Step 1, we may assume without loss of generality that $u_n(1)\to u(1)$ weakly in $H^2$ and
  $\Dt u_n(1)\to\Dt u(1)$ {\it strongly} in $L^2$. In other words, we need to pass to the limit $n\to\infty$
  in the semilinear elliptic equation
\begin{equation}\label{5.ell}
  a\Dx u_n(1)-f(u_n(1))=h_n:=\Dt u_n(1)-g.
\end{equation}
Without loss of generality we may assume also that $f'(u)\ge0$.
We will utilize the so-called energy method. Assume at this moment that we are able to integrate by
 parts and get
\begin{equation}\label{5.partint}
(f(u),\Dx u)=-(f'(u)\Nx u,\Nx u),\ \ u\in\Bbb D,
\end{equation}
this formula will be verified later at the end of the proof. Then, multiplying \eqref{5.ell}
 by $\Dx u_n$ and integrating over $x$,
we get the energy identity
\begin{equation}
(a\Dx u_n,\Dx u_n)+(f'(u_n)\Nx u_n,\Nx u_n)=(h_n,\Dx u_n).
\end{equation}
Our aim here is to pass to the limit $n\to\infty$ in this equality and compare it with the energy
 equality for the limit solution. Indeed, using the convexity arguments (similarly to \eqref{1.connvex},
  we get
\begin{multline}
 (a\Dx u(1),\Dx u(1))\le\liminf_{n\to\infty}(a\Dx u_n(1),\Dx u_n(1)),\\ (f'(u)\Nx u,\Nx u)\le
  \liminf_{n\to\infty}(f'(u_n)\Nx u_n,\Nx u_n)
\end{multline}
and due to the strong convergence $h_n\to h$, we have
$$
(h,\Dx u(1))=\lim_{n\to\infty}(h_n,\Dx u_n(1)).
$$
Then, the comparison with the limit energy identity
$$
(a\Dx u,\Dx u)+(f'(u)\Nx u,\Nx u)=(h,\Dx u(1))
$$
shows that we must have
$$
\lim_{n\to\infty}(a\Dx u_n(1),\Dx u_n(1))=(a\Dx u(1),\Dx u(1)).
$$
Together with the weak convergence $\Dx u_n(1)\to\Dx u(1)$ this gives the strong
convergence $\Dx u_n(1)\to\Dx u(1)$ in $L^2$ and, therefore, the strong
 convergence $u_n(1)\to u(1)$ in $H^2$. From the equation \eqref{5.ell} we finally establish that
  $f(u_n(1))\to f(u(1))$ also strongly. Thus, the compactness of $\hat S(1)B_R$ in $\Bbb D$ is proved. Thus, the theorem is
   proved by modulo of the integration by parts formula \eqref{5.partint} which we prove in the
   following lemma.
\end{proof}
   \begin{lemma}\label{Lem5.int} Let the nonlinearity $f$ satisfy the assumptions of
    Theorem \ref{Th5.D}. Then integration by
    parts \eqref{5.partint} is valid for every $u\in\Bbb D$.
   \end{lemma}
\begin{proof}[Proof of the lemma] We first establish the identity
\begin{equation}\label{5.smooth}
(f(u),\divv W)=-(f'(u)\Nx u, W)
\end{equation}
for all vector-fields $W\in C^\infty(\bar\Omega)$. Due to our assumption \eqref{5.fprime} both parts
 of this equality make sense. The identity may be proved by approximating the function $f$
  by "good" functions $f_n$ as in Lemma \ref{Lem1.fn}. Since $f$ has a polynomial growth, we may take, say,
$$
\Psi(z)=e^{\sqrt{z+1}}
$$
and this allows us to keep also assumption \eqref{5.fprime} uniformly in $n$. Let $u_n$ be the corresponding
 approximating functions for $u$ constructed as in \eqref{1.ell1}. Then, we first verify
 the integration by parts for $f_n$ and $u_n$ (which is trivial since everything is smooth) and after that
  pass to the limit $n\to\infty$ (which is also straightforward since as in Lemma \ref{Lem1.au}, we have
  weak convergence $f_n(u_n)\to f(u)$ in $L^2$ and, due to our assumption \eqref{5.fprime}, we also have
  weak convergence $f'_n(u_n)\Dt u_n$ to $f'(u)\Nx u$ in $L^{1+\eb}$ for small positive $\eb$. Thus, the
  integration by parts \eqref{5.smooth} is verified for smooth vector fields $W$.
  \par
Note that the $C^\infty$ smoothness assumption on the vector field $W$ can be relaxed till
$$
W\in H^1(\Omega)\cap L^\infty(\Omega)
$$
by density arguments.
\par
We now construct a sequence of Lipschitz continuous cut-off functions
\begin{equation}\label{5.cut}
\varphi_n(z)=\begin{cases}
             1,\ \ z\le n,\\ 1-\ln\frac zn,\ \ z\in[n,en],\\ 0,\ \ z>ne
            \end{cases}
\end{equation}
Then, the sequence $\varphi_n(z)$ monotone increasing in $n$ and is convergent point-wise to one.
Moreover, the
following estimate holds:
\begin{equation}\label{5.good}
|\varphi'(z)z|\le 1,\ \ z\in\R
\end{equation}
(there are no problems to construct similar {\it smooth} sequence, but we prefer to give
relatively simple
explicit expression). Then, we define a special vector-field $W=W_n$ as follows:
$$
W_n(x):=\varphi_n(|\Nx u|^2)\Nx u.
$$
Then, as simple calculation shows, $W\in L^\infty(\Omega)$ and, due to condition \eqref{5.good},
$$
\|\Nx W_n\|_{L^2}\le C\|D^2_x u\|
$$
where the constant $C$ is independent of $n$. Thus, we may conclude that
 $\divv W_n\to \Dx u$ weakly in $L^2(\Omega)$. Moreover, we may put $W_n$ to the integration by parts
  formula \eqref{5.smooth} and get
  $$
  (f(u),\divv W_n)=-(\varphi_n(|\Nx u|^2)f'(u)\Nx u,\Nx u).
  $$
It only remains to pass to the limit $n\to\infty$ here. Passing to the limit in the left-hand side is
 immediate and to pass to the limit $n\to\infty$ in the right-hand side, it is enough to note that
  $f(u)\Nx u.\Nx u$ is non-negative and belongs to $L^1(\Omega)$. The monotonicity of $\varphi_n$ in $n$ and its point-wise convergence to
   one allow us to apply the Levy monotone convergence theorem and get the desired result. Thus, the
    lemma is proved and the theorem is also proved.
\end{proof}
\begin{remark} We expect that the integration by parts formula \eqref{5.partint} holds without
 the extra assumption \eqref{5.fprime}, however, it is not clear how to verify it. Key difficulty here
  is that $\Bbb D$ is a nonlinear set and it is not easy to construct good smooth
  approximations for functions $u\in\Bbb D$.
\par
The first step in the proof of Theorem \ref{Th5.D} can be also done using the energy type arguments.
 To this
 end one just need to verify the energy identity
\begin{equation}\label{5.paren}
 \frac12\frac d{dt}\|\theta(t)\|^2_{L^2}+(a\Nx\theta(t),\Nx\theta(t))+(f'(u(t))\theta(t),\theta(t))=0
\end{equation}
which can be verified similarly to the proof of Lemma \ref{Lem5.int}. In the case of reaction
diffusion system \eqref{1.main} this is not necessary since Proposition \ref{Prop5.der} gives a simpler
 way to verify the compactness. However, it may be useful in the case of higher order equations
  where the technique of
  Proposition \ref{Prop5.der} may not work.
\end{remark}

\section{Finite dimensionality and exponential attractors}\label{s8}

In this section we discuss the finite-dimensionality of the global attractor for problem \eqref{1.main} and
 the existence of the so-called exponential attractor. We recall that a set $\Cal M\subset H$ is called an
 exponential attractor of the semigroup $\widehat S(t):H\to H$ if the following conditions are satisfied:
 \par
 1. $\Cal M$ is compact in $H$;
 \par
 2. $\Cal M$ is semi-invariant: $\widehat S(t)\Cal M\subset\Cal M$;
 \par
 3. It has a finite fractal dimension in $H$: $\dim_F(\Cal A,H)<\infty$;
\par
4. It attracts the images of bounded in $H$ sets exponentially as time tends to infinity, i.e., for
 every bounded set $B$,
$$
\dist(\hat S(t)B,\Cal M)\le Q(\|B\|_H)e^{-\alpha t}
$$
for some positive $\alpha$ and monotone function $Q$ which are independent of $B$.
\par
It is well-known that the exponential attractor if exists always contains a global attractor, so the
existence of $\Cal M$ automatically implies the finite-dimensionality of a global attractor.
In contrast to global attractors,
exponential attractors usually more robust with respect to perturbations and allow us to control
 the rate of attraction in terms of physical parameters of the system considered, but as a price to
  pay for that, an exponential attractor is not unique, see \cite{EFNT,EMZ05,MZ} for more details.
\par
The existence of an exponential attractor is usually verified using the following abstract result for
discrete semigroups $\hat S(n):=\hat S^n:H\to H$ generated by the map $\hat S:H\to H$.

\begin{proposition}\label{Prop6.a-exp} Let $H,V$ be two Banach spaces such that $V$ is compactly embedded in $H$.
Assume that there exists a bounded closed set $B\subset H$ and a map $\hat S:B\to B$ such that
\begin{equation}\label{6.sq}
\|\hat S (\xi_1)-\hat S(\xi_2)\|_V\le K\|\xi_1-\xi_2\|_H,\ \ \xi_1,\xi_2\in B.
\end{equation}
Then the corresponding discrete semigroup $\hat S(n):B\to B$ possesses an exponential attractor
 $\Cal M\subset B$.
\end{proposition}
For the proof of this proposition, see \cite{EMZ00,EMZ05}.
\par
In applications usually $B$ is an absorbing ball of the considered continuous semigroup $\hat S(t):H\to H$,
 $\widehat S:=\widehat S(T)$ for some properly chosen $T$ and \eqref{6.sq} is verified using the proper
 parabolic smoothing property for the equation on differences of two solutions. If the existence
 of a discrete exponential attractor $\Cal M_d$ is established, the exponential attractor for the continuous
  semigroup can be constructed by the standard formula:
  $$
\Cal M:=\cup_{t\in[T,2T]}\hat S(t)\Cal M_d
$$
and in order to get its finite-dimensionality, we need to assume in addition that the semigroup
 is also H\"older
continuous in time:
\begin{equation}\label{6.hol}
\|\hat S(t_1)\xi_1-\hat S(t_2)\xi_2\|_{H}\le L\(\|\xi_1-\xi_2\|_H+|t_1-t_2|^\alpha\),\ \
\end{equation}
for some $\alpha\in(0,1]$ and all $t_i\in[T,2T]$ and
$\xi_i\in B$,
see \cite{EMZ05} for the details.
\par
The main result of this section is the following theorem.

\begin{theorem} Let the nonlinearity $f$ satisfy assumptions \eqref{1.f}, \eqref{1.fp} for some
 $p\in\R_+$, \eqref{5.fprime} and the following convexity property: there exist a convex function
  $\Psi:\R^k\to\R$ such that
  \begin{equation}\label{5.con-vex}
C_2(\Psi(u)-1-|u|^2)\le |f(u)|^2\le C_1(\Psi(u)+|u|^2+1),\ \ u\in\R^k,
  \end{equation}
  for some positive constants $C_1$ and $C_2$. Let also $g\in L^2(\Omega)$ and $a$ satisfy \eqref{1.a}.
  Then problem \eqref{1.main} possesses an exponential attractor $\Cal M$ in the space $H:=L^2(\Omega)$
   which is a compact set in $\Bbb D$.
\end{theorem}
\begin{proof} According to Corollary \ref{Cor1.smm} a ball $B=B_R$ in $\Bbb D$ of a
sufficiently large radius $R$ is an absorbing set for the solution semigroup $\hat S(t):H\to H$
associated with equation \eqref{1.main}. Let us fix $T>0$ big enough that $\hat S(T)B\subset B$ and set
$\hat S:=\hat S(T)$. Then, according to estimate \eqref{1.lip}, the semigroup $\hat S(t)$ is
 Lipschitz continuous with respect to the initial data for every fixed $t$. Moreover, since $\Dt u(t)$ is
  bounded in the $L^2$-norm if $u(0)\in\Bbb D$, this semigroup is also Lipschitz continuous in time, so
  condition \eqref{6.hol}
   is satisfied with $\alpha=1$. Therefore, in order to verify the existence of an exponential attractor,
   it is enough to check the smoothing property \eqref{6.sq} for the properly chosen space $V$.
    To this end, we need to
   establish a number of smoothing estimates for the difference of solutions of equation \eqref{1.main}.
   \par
   Let $u_1(t)$ and $u_2(t)$ be two solutions of \eqref{1.main} starting from the absorbing ball $B$.
   Then their difference $\theta(t):=u_1(t)-u_2(t)$ solves the equation
   \begin{equation}\label{6.dif}
   \Dt\theta=a\Dx\theta-l(t)\theta,\ \ l(t):=\int_0^1f'(su_1(t)+(1-s)u_2(t))\,ds.
   \end{equation}
   We recall that, due to the assumption $f'(u)\ge-K$, multiplication of this equation on $\theta$
   gives the basic Lipschitz
    continuity estimate:
\begin{equation}\label{6.l2}
     \|\theta(T)\|^2_{L^2}+\int_0^T\|\Nx \theta(t)\|^2_{L^2}\,dt\le Ce^{KT}\|\theta(0)\|_{L^2}^2,
\end{equation}
see Lemma \ref{Lem1.lip}. Moreover, multiplying \eqref{6.dif} by $\theta|\theta|^r$ and arguing
exactly as in the proof of Proposition \ref{Prop5.der},
we get the estimate
\begin{equation}\label{6.lr-sm}
\|\theta(T)\|_{L^{r+2}}\le Ce^{KT}T^{-1}\|\theta(0)\|_{L^2}
\end{equation}
for some sufficiently small positive $r$ depending only on the matrix $a$.
\par
In order to get smoothing estimate for $\theta$, we argue as in Step 1 of the proof
 of Theorem \ref{Th5.D}. Namely, from \eqref{5.fprime} and \eqref{5.con-vex}, we conclude that
\begin{multline}
 |f'(su_1+(1-s)u_2)|^2\le\\\le C(|f(s u_1+(1-s)u_2)|^2+1+|u_1|^2+|u_2|^2)\le\\\le
  C'(|f(u_1)|^2+|f(u_2)|^2+|u_1|^2+|u_2|^2+1),\ s\in(0,1)
\end{multline}
and, therefore, since $f(u(t))$ is uniformly bounded in $L^2$-norm for our solutions $u_1$ and
$u_2$, we have
\begin{equation}
\|l(t)\|_{L^2}\le C,\ \ t\in[0,T].
\end{equation}
This estimate, in turn, implies (together with \eqref{6.l2}, Sobolev embedding theorem and
H\"older inequality)
 that
 $$
 \|l(t)\theta\|_{L^s(0,T;L^s)}\le C_T.
 $$
 for some $1<s<2$. Applying now the $L^s$ interior regularity estimate to equation
 \eqref{6.dif} and arguing as in the proof
  of Theorem \ref{Th5.D}, we get
  $$
  \|\theta(T)\|_{W^{2(1-\frac1s),s}}\le C_T\|\theta(0)\|_{L^2}
  $$
which together with the embedding \eqref{5.int-in} gives
$$
\|\theta(T)\|_{H^\eb}\le C_T\|\theta(0)\|_{L^2}
$$
for some positive exponent $\eb$. Setting finally $V=H^\eb(\Omega)$ we get the desired smoothing property
\eqref{6.sq} and finish the proof of the theorem.
\end{proof}

\begin{remark} The finite-dimensionality of the global attractor $\Cal A$ has been established under
 similar
assumptions on $f$ in \cite{Z1} using the so-called method of $l$-trajectories, see also \cite{MN96,MP02}.
 In
 the present work we suggest the simplified version of the proof which is based on multiplication of equation
  \eqref{6.dif} on the quantities like $\theta|\theta|^r$. Although the proof becomes more transparent,
  it is
   slightly less general than the one suggested in \cite{Z1} since this multiplication is suitable for
    reaction-diffusion systems and may not work for more general
     (e.g., higher order equations). In such cases one should return
     back to the method of $l$-trajectories.
\end{remark}

\section{Generalizations and concluding remarks}\label{s9}
In this concluding section we briefly consider other types of equations for which the technique
developed above works (with some minor changes which we will discuss) and state some interesting
 open problems. We start with the case of fractional Laplacians and the corresponding
 reaction-diffusion equations which are becoming more and more popular nowadays,
 see \cite{ASS16,AMR,GW16,LSS} and references therein for more details.

\subsection{Fractional reaction-diffusion systems} Let us define $A:=(-\Dx)^{\alpha}$, $0<\alpha<1$,
in the domain $\Omega$ endowed with Dirichlet boundary conditions and consider the following fractional
 reaction-diffusion system:
 \begin{equation}\label{9.frds}
 \Dt u+a(-\Dx)^\alpha u +f(u)=g,\ \ u\big|_{\partial\Omega}=0,
 \end{equation}
where the function $f$ and the matrix $a$ satisfy assumptions \eqref{1.f} and \eqref{1.a} respectively.
In this case, the definition of the phase space $\Bbb D$ should be reduced as follows:
\begin{equation}\label{9.Dal}
\Bbb D_\alpha:=\{u\in H^{2\alpha}_{\Dx},\ \ f(u)\in L^2(\Omega)\},
\end{equation}
where $H^{2\alpha}_{\Dx}:=D((-\Dx)^\alpha)$.
\par
All of the estimates and results stated above for the case $\alpha=1$ can be extended in a
 straightforward way to a general case $0<\alpha<1$. The only non-trivial place is the
  estimates of the terms like $((-\Dx)^\alpha u, f(u))$ or $((-\Dx)^\alpha u,u|u|^r)$.
  In the case when $\Omega=\R^d$
   (or in the case of periodic BC), we have a nice explicit formula for such inner products
   which trivializes the required estimates (see e.g., \cite{tri}), namely,
   \begin{equation}\label{9.nice}
(Au,v)=C_\alpha\int_{\R^d}\int_{\R^d}\frac{(u(x)-u(y))(v(x)-v(y))}{|x-y|^{d+2\alpha}}\,dx\,dy.
   \end{equation}
In particular, it gives the positivity of $(Au,f(u))$ if $f'(u)\ge0$. Fortunately, there is an extension
of this formula to the case of bounded domains (see \cite{Caf16}), namely,
\begin{multline}
(Au,v)=\int_\Omega\int_\Omega (u(x)-u(y))(v(x)-v(y))K_{\Omega,\alpha}(x,y)\,dx,\,dy+\\+
\int_\Omega u(x)v(x)B_{\Omega,\alpha}(x)\,dx
\end{multline}
for some non-negative functions $K_{\Omega,\alpha}$ and $B_{\Omega,\alpha}$. This
formula allows us to get the
 same type of estimates as for the local case $\alpha=1$.  For the convenience of the reader,
 we state below
  the analogues of two main results for the fractional case.
\begin{theorem}\label{Th9.crit} Let the matrix $a$ and the nonlinearity $f$ satisfy
\eqref{1.a} and \eqref{1.f} respectively  and let,
 in addition,  the nonlinearity
  $f$ satisfy \eqref{1.fp} with the exponent $p$ restricted by the assumption
  \begin{equation}
p<p_{crit}+\eb,\ \ p_{crit}:=1+\frac{4\alpha}{d-4\alpha},\ \eb=\eb(a)>0
  \end{equation}
  if $d\ge4\alpha$
  and the external forces $g\in L^q(\Omega)$ for some $q>\frac d{2\alpha}$. Then any
   weak solution $u(t)$ of
   problem \eqref{9.frds} starting from $u(0)\in H$ possesses the following smoothing property:
   \begin{equation}
   \|u(t)\|_{L^\infty}\le Q_t(\|u(0)\|_{L^2}+\|g\|_{L^q}),\ \ t\in(0,1]
   \end{equation}
   for some monotone function $Q_t$ depending on $t$, but  independent of $u$ and $g$.
\end{theorem}
\begin{remark} Note that this result is not very helpful if $0<\alpha<\frac12$ since the direct
$H^1$-estimate
which is obtained by multiplication of the equation by $-\Dx u$ gives the control of the $H^1$-norm
(of course, assuming in addition that $g\in H^{1-\alpha}$) which
 is better than $H^{2\alpha}$-control finally obtained from $u\in\Bbb D_\alpha$. However, it is useful
  for $\alpha\ge\frac12$. In particular, in the case $0< \alpha<\frac34$, we may have the
   supercritical growth rate in the case of physical dimension $d=3$ as well. So, main results
    become applicable for $d=3$ as well. Note also that many of the results of our paper
     may be extended also
     to the case $\alpha>1$ (e.g., to the Swift-Hohenberg type equations where $\alpha=2$),
      but in this case we will be not able to multiply the equation by $Au$ since the term $(Au,f(u))$
       will be out of control, so we may multiply it only on $\Dx u$ and this gives the control
        of the $H^{\frac{1+\alpha}2}$-norm of $u(t)$ (not $H^{2\alpha}$ as before). Moreover, in this
         case $\eb(a)=0$ since multiplication on $u|u|^r$ is no more available.
\end{remark}
We now state the key result about exponential attractors for the supercritical case.

\begin{theorem}\label{Th9.attr} Let the nonlinearity $f$ satisfy assumptions \eqref{1.f},
 \eqref{1.fp} for some
 $p\in\R_+$, \eqref{5.fprime} and the following convexity property: there exist a convex function
  $\Psi:\R^k\to\R$ such that
  \begin{equation}
C_2(\Psi(u)-1-|u|^2)\le |f(u)|^2\le C_1(\Psi(u)+|u|^2+1),\ \ u\in\R^k,
  \end{equation}
  for some positive constants $C_1$ and $C_2$. Let also $0<\alpha<1$, $g\in L^2(\Omega)$ and $a$
   satisfy \eqref{1.a}.
  Then problem \eqref{9.frds} possesses an exponential attractor $\Cal M$ in the space $H:=L^2(\Omega)$
   which is a compact set in $\Bbb D_\alpha$.
\end{theorem}

\subsection{Cahn-Hilliard type systems} Let us consider the following fractional Cahn-Hilliard
 system in $\Omega\subset\R^d$:
\begin{equation}\label{9.CH}
\Dt u+(-\Dx)^\beta(a(-\Dx u)^\alpha u+f(u)-g)=0
\end{equation}
endowed by the Dirichlet boundary conditions. We assume here that $0<\beta\le1$, $0<\alpha\le1$. Note that $\alpha=\beta=1$ corresponds to the
 classical Cahn-Hilliard system and $\beta=0$, $\alpha=1$ to the reaction-diffusion system considered above.
 See \cite{ASS16,MZ,tem} and references therein for more details concerning classical and
  fractional CH-equations.
  It is natural to take $\Bbb D_\alpha$ as the phase space for this problem and rewrite it
   in the following form:
  \begin{equation}
\Dt (-\Dx)^{-\beta} u+a(-\Dx)^\alpha u+f(u)=g.
  \end{equation}
Then we may utilize the monotonicity of the function $f$ and apply the developed above theory
 to this equation (see also \cite{MZ} for the case $\alpha=\beta=1$). In this case, weak solutions are
  naturally defined in the space $H:=H^{-\beta}(\Omega)$ and strong solutions live in $\Bbb D_\alpha$.
  \par
  The key result on the existence of exponential attractors now reads.
  \begin{theorem} Let the nonlinearity $f$ satisfy assumptions \eqref{1.f}, \eqref{1.fp} for some
 $p\in\R_+$, \eqref{5.fprime} and the following convexity property: there exist a convex function
  $\Psi:\R^k\to\R$ such that
  \begin{equation}
C_2(\Psi(u)-1-|u|^2)\le |f(u)|^2\le C_1(\Psi(u)+|u|^2+1),\ \ u\in\R^k,
  \end{equation}
  for some positive constants $C_1$ and $C_2$. Let also $g\in L^2(\Omega)$ and $a$ satisfy \eqref{1.a}.
  Then problem \eqref{9.CH} possesses an exponential attractor $\Cal M$ in the space
   $H:=H^{-\beta}(\Omega)$ which is a compact set in $\Bbb D_\alpha$.
\end{theorem}
We leave the rigorous proof of this theorem to the reader.

\subsection{Open problems} We conclude this section by a discussion of some open questions and
 possible further improvements of the above developed theory.
\par
{\bf Problem 1.} We start with the already posed question about the validity of the integration
 by parts formula
\begin{equation}
(f(u),\Dx u)=-(f'(u)\Nx u,\Nx u)
\end{equation}
for every $u\in \Bbb D$. We know that both parts of this equality are well-defined for any $u\in\Bbb D$.
 However, since we do not know the density of smooth functions in $\Bbb D$, we cannot verify the identity
  in a standard way, so we need to use something else. We have proved this identity under the extra
   assumption \eqref{5.fprime} which allows us to control the Lebesgue norm of $f'(u)\Nx u$ and
    simplifies the situation. Clarifying the situation with this integration by parts in general would be
    very useful for establishing energy equalities for many other equations containing monotone
     nonlinearities which, in turn, may give compactness of the corresponding global attractors.
     We were sure
      that \eqref{5.fprime} is technical, but surprisingly are unable to remove
      it (or find the  proper reference).
\par

{\bf Problem 2.} Next problem is related with smoothness of weak/strong
solutions of problem \eqref{1.main}. We have established that under the assumption \eqref{1.fp} that $f$ has
 a polynomial growth rate, the problem possesses an instantaneous smoothing $H$ to $\Bbb D$ smoothing
  property. It would be interesting to understand whether or not this polynomial growth restriction is
   really necessary for the smoothing (ideally, to construct a non-smoothing
    weak solution for problem \eqref{1.main}), say, with exponential or stronger nonlinearities.
     A natural idea here is to extend the proof of Theorem \ref{Th1.strange} to the case
      where $\Dt u$ belongs to some
     weaker spaces than $L^p(\Omega)$ with $p\ll1$ using the technique of Orlich spaces. But more
     detailed analysis shows that this does not work already when $\ln (1+|f(u)|)\in L^1$, so we may
     expect the existence of such exotic non-smoothing weak solutions for fast growing nonlinearities.
     \par
     The phenomenon of delayed regularization is well-known in the class of nonlinear
      diffusion problems, see \cite{Va} and reference therein. For example, the equation
      $$
      \Dt u|\Dt u|^p=\Dx u, \ u\big|_{\partial\Omega}=0, \ \ p\ge0
      $$
      is well-posed in a natural energy phase space $\Phi=W^{1,2}_0(\Omega)$. However, the
      solutions of this equation {\it do not possess} an instantaneous smoothing if, say, $p>4$ and $d=3$.
      Indeed, the energy identity for this equation reads
      $$
      \|\Nx u(T)\|_{L^2}^2+\int_0^T\|\Dt u(t)\|^{p+2}_{L^{p+2}}\,dt=\|\Nx u(0)\|^2_{L^2},
      $$
      so if $u(0)\notin L^{p+2}(\Omega)$, we  have $u(T)\notin L^{p+2}(\Omega)$ for any finite $T>0$.
      However, if we start from more regular phase space $\Psi:=W^{1,2}_0(\Omega)\cap L^\infty(\Omega)$,
      we will have instantaneous further regularization, see \cite{EZel}. The open question is
      whether or not something similar happens in the case of system \eqref{0.rds} of reaction-diffusion
       equations with
      fast growing nonlinearity $f$ satisfying \eqref{1.f}.
     \par
     Another related question is about generating singularities in finite time in equations
      like \eqref{1.main}. It is known that general reaction-diffusion systems may generate
       singularities in higher norms even if the natural energy norm  remains finite and dissipative, see
       e.g., \cite{P00} for RDS satisfying balance law (=action mass law), \cite{HV97} for the case of
        reaction-diffusion with chemotaxis or \cite{Bud} for Ginzburg-Landau equations
         in $\R^3$ (see also references therein). However, to the best of our knowledge, there are no
          such examples in the class of equation \eqref{1.main} with nonlinearities
          satisfying $f'(u)\ge-K$.
          As we know, in this case the $H^2$-norm cannot blow up, so this is the question of
          possible blow up of higher norms and high space dimension $d>4$.

\par
{\bf Problem 3.} Finally, about the finite-dimensionality of global attractors. The most popular
 scheme for proving this result is related with volume contraction technique,
 see \cite{BV,tem} and references therein. Using this technique, we need to estimate $l$ dimensional traces
 $\operatorname{Tr}_l\Cal L_u$, where
 $$
 \Cal L_uv=a\Dx v-f'(u)v
 $$
 is the linearized operator on the trajectory $u(t)$ of the equation \eqref{1.main}
  lying on the attractor. {\it Formal} estimates of this quantity depend only on $K$ (if the assumption
   $f'(u)\ge-K$ is posed) and are independent on the norm of $u(t)$ and any norms of $f(u)$.
   \par
   However, to justify this method we need to verify the {\it differentiability} of the semigroup
    $\hat S(T)$ with respect to the initial data (at least the so-called uniform quasi-differentiability
     on the attractor, see \cite{tem}) and such a differentiability usually {\it does not}
      hold in supercritical cases.
\par
This was the main reason to use the alternative scheme based on Proposition \ref{Prop6.a-exp} for verifying
 the finite-dimensionality. In this scheme the differentiability is not required, but as the price to pay,
 we get essentially worse estimates than expected since now the norm of $|f'(u)|$ is involved into all
  dimension estimates.
\par
It would be interesting to remove this drawback and remove the dependence on $|f'(u)|$ from these estimates,
 e.g., by finding a "clever" choice of spaces $H$ and $V$ in Proposition \ref{Prop6.a-exp}. Up to the
  moment we know how to do this in a scalar case only, due to the possibility to multiply \eqref{6.dif} by
  $\sgn v$ and using the Kato inequality. This in turn gives the estimate of the
   $L^1$-norm of $l(t)\theta$ through quantities depending only on $K$. To the best of our knowledge,
   nothing similar is known for the vector case.

\appendix
\section{Nonlinear localization and elliptic regularity}\label{A}

In this appendix we consider the following semi-linear elliptic problem:
\begin{equation}\label{A.e}
a\Dx u-f(u)=g,
\end{equation}
where the matrix $a$ satisfies assumption \eqref{1.a} and $f$ enjoys assumptions \eqref{1.f}. Then, arguing
as before, we get the $H^2$-elliptic regularity
\begin{equation}\label{A.2reg}
\|u\|_{H^2}+\|f(u)\|_{L^2}\le C\|g\|_{L^2}.
\end{equation}
The question addressed here concerns an additional  regularity under the extra assumption
\begin{equation}\label{A.gq}
g\in L^q(\Omega),\ \ q>2.
\end{equation}
A partial answer on this question is given in the following theorem.

\begin{theorem}\label{ThA.main} Let the above assumptions hold. Then there exists
 $\kappa=\kappa(a)>0$ such that
\begin{equation}\label{A.strest}
\||D^2_xu|\cdot|\Nx u|^{r/2}\|^2_{L^2}+\|\Nx u\|^{r+2}_{L^{\frac{d(r+2)}{d-2}}}\le C\|g\|_{L^q}^{r+2}
\end{equation}
if $d>2$, $q<d-\frac{d(d-2)}{\kappa+d}$ and $r=\frac{d(q-2)}{d-q}$. Here and below $D^2_xu$ stands
 for the collection of all
 second derivatives of the function $u$.
\end{theorem}
\begin{proof}We give below only the formal derivation of estimate \eqref{A.strest} which can be justified
 in a standard way (e.g., by cutting off the nonlinearity as explained in section \ref{s.s}, mollifying
  $g$   and using
 the corresponding smooth  solutions of the cut off equation to approximate the initial solution $u$).
 \par
 {\it Step 1.} We start with the simplest case of periodic boundary conditions where no difficult terms
 related with the boundary arise and we may integrate by parts freely. Without loss of generality,
 we may also
  assume that $f'(u)\ge0$ and $f(0)=0$.
\par
  Let us multiply equation
  \eqref{A.e} by $\partial_{x_i}(\partial_{x_i}u|\Nx u|^r)$ and integrate over $x$. This gives
  \begin{multline}\label{A.del}
(a\Dx u,\partial_{x_i}^2u|\Nx u|^r)+(f'(u)\partial_{x_i}u,\partial_{x_i}u|\Nx u|^r)\le\\\le
(g,\partial_{x_i}(\partial_{x_i}u|\Nx u|^r))+Cr(|D^2_xu|^2,|\Nx u|^r).
  \end{multline}
Integrating by parts twice in the first term in the left-hand side and using the positivity of $a$ and $f'$,
we arrive at
\begin{equation}\label{A.divi}
\alpha(|\Nx (\partial_{x_i}u)|^2,|\Nx u|^r)
\le
(g,\partial_{x_i}(\partial_{x_i}u|\Nx u|^r))+Cr(|D^2_xu|^2,|\Nx u|^r).
\end{equation}
Due to Sobolev embedding theorem, we have
\begin{equation}\label{A.Sol}
\|\Nx u\|_{L^{\frac{d(r+2)}{d-2}}}^{r+2}\le C\|\Nx(|\Nx u|^{\frac{r+2}2})\|^2_{L^2}\le
 C'(|D^2_x u|^2,|\Nx u|^r).
\end{equation}

Taking now a sum in \eqref{A.divi} with respect to $i=1,\cdots, d$ and assuming that
 $r\le \kappa=\kappa(\alpha)$
is small enough, we end up with
\begin{multline}
\alpha\((|D^2_x u|^2,|\Nx u|^r)+\|\Nx u\|_{L^{\frac{d(r+2)}{d-2}}}^{r+2}\)\le
\\\le C|(g,|D^2_x u||\Nx u|^{r/2}|\Nx u|^{r/2})|.
\end{multline}
Finally, using the H\"older inequality with the exponents $q$, $2$ and $\frac{2d(r+2)}{r(d-2)}$, we get
\begin{multline}
|(g,|D^2_x u|^2|\Nx u|^{r/2}|\Nx u|^{r/2})|\le\\\le \frac\alpha2\((|D^2_x u|^2,|\Nx u|^r)+
\|\Nx u\|_{L^{\frac{d(r+2)}{d-2}}}^{r+2}\)+C_\alpha\|g\|^{r+2}_{L^q}
\end{multline}
and finish the proof of the theorem in the case of periodic BC.
\par
The case of Dirichlet BC is more delicate since hardly controllable boundary terms will appear if
 we try to integrate by parts in the first term of \eqref{A.del}. To avoid them, we will use the nonlinear
 localization technique suggested in \cite{KZ09} (see also \cite{KZ12}) for more details.
 \par
 {\it Step 2. Interior estimates.} We introduce a cut-off function $\theta\in C^2(\R^d)$ which vanishes
 near the boundary and equals to one in the $\delta$-interior of the domain $\Omega$ ($\delta\ll1$
 is small positive) and satisfies
 $$
 |\Nx\theta(x)|\le C\theta(x)^{1/2},
 $$
such a function exists and at least $C^1$-smooth since the domain is smooth. Then, multiplying equation
 \eqref{A.e} by $\Nx\cdot(\theta\Nx u|\Nx u|^r)$ and arguing as in Step 1, we arrive at
 \begin{equation}\label{A.int}
(\theta|D^2_x u|^2,|\Nx u|^r)\le C_\delta\(\|g\|^{r+2}_{L^q}+\|\Nx u\|^{r+2}_{L^{r+2}}\).
 \end{equation}
 Indeed, integration by parts is now allowed due to the factor $\theta$ vanishing near the boundary.
 Of course this will produce the extra terms containing $\Nx\theta$ at every step, but these terms
 are all under
  the control due to obvious estimate:
\begin{multline}\label{A.LOT}
|(|D^2_x u||\Nx u|,|\Nx\theta|,|\Nx u|^r)|\le \eb(\theta|D^2_x u|^2,|\Nx u|^r)+\\+C_\eb(|\Nx\theta|^2\theta^{-1},|\Nx u|^{r+2})\le
\eb(\theta|D^2_x u|^2,|\Nx u|^r)+C_\eb\|\Nx u\|^{r+2}_{L^{r+2}},
\end{multline}
where $\eb>0$ is arbitrary.
\par
{\it Step 3. Boundary estimates: tangential derivatives.} To treat the neighbourhood of the boundary,
 we introduce
 an $x$-depending smooth orthonormal base $(\tau_1(x),\cdots,\tau_{d-1}(x),n(x))$ near the boundary
  such that, when
 $x\in\partial\Omega$, $\tau_i$ correspond to {\it tangential} directions and $n(x)$ is outer normal. We also
  assume that these vector fields are cut-off outside of small neighbourhood of the boundary
   similarly to Step 2. Note also that in general such smooth vector fields exist near the
    boundary only locally, but we will ignore this fact assuming that they exist globally
     (one more localization is necessary in general). We also define the corresponding directional
 derivatives
 $$
 \partial_{\tau_i}:=\tau_i.\Nx,\ \ \partial_n:=n.\Nx.
 $$
 Then, as follows from the orthogonality conditions,
 $$
 |\Nx u|=|\nabla_{(\tau,n)}u|,\ \ [\partial_{\tau_i},\partial_{\tau_j}]=L.O.T.,\
  \ \partial_{\tau_i}^*=-\partial_{\tau_i}+L.O.T.
 $$
 where "L.O.T." means "lower order terms" and also we have analogous commutator formulas
  for normal derivative as well. In addition,
\begin{multline}
\Dx u=\partial_{n}^2u+\sum_{i=1}^{d-1}\partial_{\tau_i}^2u+L.O.T.,\\ \Nx\cdot(\Nx u|\Nx u|^r)=
\nabla_{(\tau,n)}\cdot(\nabla_{(\tau,n)}u|\nabla_{(\tau,n)}|^r)+L.O.T.
\end{multline}
Since all "L.O.T." are under the control analogously to estimate \eqref{A.LOT}, we may replace
$x$-derivatives by $(\tau,n)$-derivatives and do calculations simply assuming that they  commute.
At this step we multiply equation \eqref{A.e} by
$-\partial_{\tau_i}^*(\partial_{\tau_i}u|\nabla_{(\tau,n)}u|^r)$ and integrate over $x$. Analogously to Step 1, this gives
\begin{multline}\label{A.tan}
(a\partial_{n}^2u+a\sum_{j=1}^{d-1}\partial_{\tau_j}^2 u,\partial_{\tau_i}^2u|\nabla_{(\tau,n)}u|^r)
+(f'(u)\partial_{\tau_i}u,\partial_{\tau_i}u|\Nx u|^r)
\le\\\le
\eb(|D^2_xu|^2,|\Nx u|^r)+C_\eb\(\|g\|^{r+2}_{L^q}+\|\Nx u\|^{r+2}_{L^{r+2}}\).
\end{multline}
The advantage of separating tangential and normal derivatives is the fact that
$$
\partial_{\tau_i}u\big|_{\partial\Omega}=\partial_{\tau_i}^2u\big|_{\partial\Omega}=0
$$
due to the Dirichlet boundary conditions, so we again may integrate by parts the expression in the left-hand side
 of \eqref{A.tan} freely and, analogously to Step 1, get the following estimate
  (using also that $f'(u)\ge0$):
\begin{multline}\label{A.good-tan}
(|D^2_{\tau,\tau}u|^2+|D^2_{n,\tau} u|^2,|\Nx u|^r)\le\\\le \eb(|D^2_xu|^2,|\Nx u|^r)+
C_\eb\(\|g\|^{r+2}_{L^q}+\|\Nx u\|^{r+2}_{L^{r+2}}\).
\end{multline}
{\it Step 4. Boundary estimates: normal derivatives.} To estimate normal derivatives we multiply
\eqref{A.e} by $\partial_n^*(\partial_n u|\Nx u|^r)$ and use that $f(0)=0$ in order to kill boundary terms
 related with nonlinearity. Together with the already obtained estimate \eqref{A.good-tan} this gives
\begin{multline}\label{A.norm}
(\partial_n^2 u,\partial_n^2 u|\Nx u|^r)\le \eb(|D^2_xu|^2,|\Nx u|^r)+
C_\eb\(\|g\|^{r+2}_{L^q}+\|\Nx u\|^{r+2}_{L^{r+2}}\)\\+|(|D^2_{\tau,\tau}u|,|\partial_n^2 u||\Nx u|^r)|\le
\eb(|D^2_xu|^2,|\Nx u|^r)+
C_\eb\(\|g\|^{r+2}_{L^q}+\|\Nx u\|^{r+2}_{L^{r+2}}\),
\end{multline}
where the constant $\eb>0$ can be chosen arbitrarily small.
\par
{\it Step 5. Combining all together.} Combining the interior estimates obtained at Step 2 with the tangential
 and normal boundary estimates \eqref{A.good-tan} and \eqref{A.norm} (e.g., with the help of the proper
  partition of unity), we finally arrive at
 $$
(|D^2_x u|^2,|\Nx u|^r)\le C_\eb(\|g\|^{r+2}_{L^{r+2}}+\|\Nx u\|^{r+2}_{L^{r+2}})+C\eb(|D^2_xu|,|\Nx u|^r)
$$
and after fixing $\eb>0$ small enough, we end up with
$$
(|D^2_x u|^2,|\Nx u|^r)\le C(\|g\|^{r+2}_{L^{r+2}}+\|\Nx u\|^{r+2}_{L^{r+2}}).
$$
So, it only remains to estimate the $L^{r+2}$-norm of the gradient $\Nx u$. To this end, it is enough to use
 \eqref{A.Sol} together with the obvious estimate
 $$
 \|\Nx u\|^{r+2}_{L^{r+2}}\le \eb\|\Nx u\|^{r+2}_{L^{\frac{d(r+2)}{d-2}}}+C_\eb\|\Nx u\|^{r+2}_{L^2}
 $$
 and estimate \eqref{A.2reg}. This completes the proof of the theorem.
\end{proof}
\begin{remark} As we see, the nonlinear localization uses the general strategy of the classical
 (linear) localization technique. However, it is more delicate since we need also to treat the
  nonlinear term $f(u)$ which is now {\it not subordinated} to the linear ones, so we can multiply the
   equation only on the terms which can be estimated using the monotonicity assumption $f'(u)\ge0$.
    Fortunately, the amount of such multipliers is enough to get the estimates similar to the linear case.
\end{remark}

\end{document}